\documentclass[11pt,letterpaper,hidelinks]{article}
\usepackage[utf8]{inputenc}
\usepackage[T1]{fontenc}
\usepackage{amsmath}
\usepackage{amsfonts}
\usepackage{amssymb}
\usepackage{amsthm}
\usepackage{graphicx}
\usepackage{hyperref}
\usepackage{float}
\usepackage{caption}
\usepackage{subcaption}
\usepackage{amssymb}
\usepackage{mathrsfs}

\newcommand{\eq}{:=}

\newcommand{\grad}{\boldsymbol \nabla}
\renewcommand{\div}{\grad \cdot}

\newcommand{\ddiv}{\operatorname{div}}

\newcommand{\jmp}[1]{\,[\![#1]\!]}

\newcommand{\BA}{\boldsymbol A}

\newcommand{\BL}{\boldsymbol L}

\newcommand{\bn}{\boldsymbol n}

\newcommand{\bp}{\boldsymbol p}
\newcommand{\bq}{\boldsymbol q}

\newcommand{\bv}{\boldsymbol v}

\newcommand{\bz}{\boldsymbol z}

\newcommand{\CF}{\mathcal F}

\newcommand{\CH}{\mathcal H}

\newcommand{\CR}{\mathcal R}

\newcommand{\CT}{\mathcal T}

\newcommand{\CV}{\mathcal V}

\newcommand{\BCH}{\boldsymbol{\CH}}

\newcommand{\BCV}{\boldsymbol{\CV}}

\usepackage[top=3cm,bottom=2cm,left=2cm,right=2cm]{geometry}

\hypersetup{colorlinks, linkcolor=blue, citecolor=blue,urlcolor=blue,plainpages=false}

\newtheorem{theorem}{Theorem}[section]
\newtheorem{lemma}[theorem]{Lemma}
\newtheorem{corollary}[theorem]{Corollary}
\newtheorem{assumption}[theorem]{Assumption}
\newtheorem{remark}[theorem]{Remark}
\newtheorem{proposition}[theorem]{Proposition}

\usepackage{todonotes}
\usepackage{xcolor}
\usepackage{cancel}
\usepackage{empheq}

\DeclareMathOperator*{\argmin}{arg\,min}
\usepackage{authblk}

\usepackage{cleveref}

\title{An adaptive superconvergent finite element method based on local residual minimization\footnotetext{{\bf Funding:} IM acknowledges support from the project DI Investigación Innovadora Interdisciplinaria PUCV No. 039.409/2021, and project DI PUCV No. 039.375/2021. The work of SR was supported by the Chilean grant ANID FONDECYT No. 3210009. The research by PV was supported by the Chilean grants DI Postdoctorado 2021 (PUCV) and ANID FONDECYT No. 3220858. Finally, the authors have also received funding from the European Union’s Horizon 2020 research and innovation programme under the Marie Sklodowska-Curie grant agreement No. 777778 (MATHROCKS).}}
\author[1,2]{Ignacio Muga}
\author[1]{Sergio Rojas}
\author[1]{Patrick Vega}
\affil[1]{Instituto de Matem\'aticas, Pontificia Universidad Cat\'olica de Valpara\'{\i}so, Valpara\'{\i}so, Chile
	\texttt{\{ignacio.muga;sergio.rojas.h;patrick.vega\}@pucv.cl}}
\affil[2]{BCAM - Basque Center for Applied Mathematics, Bilbao, Spain}

\begin{document}
	\maketitle
	
	\begin{abstract}
We introduce an adaptive superconvergent finite element method for a class of mixed formulations to solve partial differential equations involving a diffusion term. It combines a superconvergent postprocessing technique for the primal variable with an adaptive finite element method via residual minimization. Such a residual minimization procedure is performed on a local postprocessing scheme, commonly used in the context of mixed finite element methods. Given the local nature of that approach, the underlying saddle point problems associated with residual minimizations can be solved with minimal computational effort. We propose and study a posteriori error estimators, including the built-in residual representative associated with residual minimization schemes; and an improved estimator which adds, on the one hand, a residual term quantifying the mismatch between discrete fluxes and, on the other hand, the interelement jumps of the postprocessed solution. We present numerical experiments in two dimensions using Brezzi-Douglas-Marini elements as input for our methodology. The experiments perfectly fit our key theoretical findings and suggest that our estimates are sharp.
	\end{abstract}
	
	
	
	\setcounter{tocdepth}{3}
	\tableofcontents
	
\section{Introduction}
Finite Element Methods (FEMs) constitute one of the most popular approaches for solving boundary value problems driven by partial differential equations (PDEs). FEMs are suited to handle heterogeneous media and/or complex geometries, mainly due to the great flexibility provided by unstructured meshes. In this paper, we will focus on PDEs ruled by diffusion phenomena, where the flux $\bq$ of a concentration $u$ is described by the equation $\bq=-\grad u$ (i.e., the flux will move from areas of higher concentration to areas of lower concentration). Mixed finite element methods \cite{boffi2013mixed,ern_guermond_2021b} are suitable to solve such PDEs when the dual variable $ \bq $ is required as a fundamental unknown\footnote{In contrast to primal-based methods, where this variable must be obtained a posteriori by differentiation, entailing a loss of accuracy \cite{arnold}.}.

Many postprocessing techniques \cite{ainsworth_ma_2012,arnold_brezzi_1985,bramble_xu_1989,stenberg_1991} have been offered to obtain better approximations for the primal variable $  u $ in the context of mixed finite element methods. The main idea is to enhance the approximation of $u$ using minimal computational effort (hence, by performing local computations at the element level, which can be efficiently parallelized).

We will focus on the postprocessing technique proposed by Stenberg \cite{stenberg_1991} for its conceptual simplicity. This postprocessing method (and those closely related to it) has been successfully employed in conjunction with different discretization schemes to solve a wide range of PDEs, illustrating its versatility. Among them, we highlight mixed methods \cite{MR2801446}, Discontinuous Galerkin (DG) methods \cite{MR2448694}, Discontinuous Petrov-Galerkin (DPG) methods \cite{MR3977484} and mimetic finite difference methods \cite{MR2365823}. In particular, a large number of works have used the combination of Hybridizable DG (HDG) methods with the Stenberg's postprocessing technique on a wide variety of equations\footnote{See, e.g., \cite{MR4246271} for a exhaustive review of the state-of-the-art on HDG methods, most of them employing Stenberg's postprocessing.}.

Regarding practical implementations of FEMs, mesh (or $ p\ \!$\!-) adaptivity has become an attractive strategy to limit the computational cost required for the discrete resolution of boundary value problems. Such a local adaptation procedure is commonly guided by an iterative process, driven by an a posteriori error estimator, which quantitatively \emph{measure} the discretization error, letting practitioners decide if the accuracy of the numerical solution is enough for their applications. Consequently, the design and study of a posteriori error estimators have been a field of great interest in the last 40 years (see, e.g., \cite{ainsworth_oden_2000a,demkowicz_2006a,verfurth_1994} and the references therein). An alternative to the design of a posteriori error estimators is the use of finite element methods based on residual minimization, which tend to incorporate a built-in localizable \emph{error representative} that may be used to quantify the error and drive adaptive mesh refinements. The common procedure uses a mixed form to simultaneously solve the discrete solution and its \emph{error representative}.

In this paper, we combine the local postprocessing idea with the concept of residual minimization to come up with a local postprocessing method that delivers, at the same time, a superconvergent postprocessed (scalar) solution and a superconvergent local error representative that is used to drive adaptive mesh refinements. The mixed system related to the residual minimization is local (i.e., element-by-element) and therefore, the additional computational efforts to achieve superconvergence are negligible. We show that our postprocessed solution coincides with the one proposed by Stenberg and therefore satisfies the same a priori error estimates. Besides, we propose and study a posteriori error estimators based on the built-in residual representative associated with residual minimization schemes.

The methodology is quite general and can be applied to any problem exhibiting the equation $\bq=-\grad u$ in its mixed form. This may include the presence of convective or reactive terms. Moreover, we can incorporate variable diffusion coefficients or deal with vectorial equations (such as Stokes-like problems). The design of our estimators is (almost) independent of the underlying PDE (since it only depends on the constitutive equation $ \bq=-\grad u $) and the initial choice of finite element spaces discretizing the mixed form.

We illustrate our theoretical findings through: the behavior of effectivity indices for uniform mesh refinements in the case of smooth solutions; the capability of our estimators to drive adaptive mesh refinements in the presence of singularities (recovering optimal convergence rates); and the application of our methodology to convection-diffusion problems.
\subsection{A review of related literature}
a posteriori error estimators employing postprocessing techniques may be found in the works \cite{MR2914423,MR3082572,MR2365826,MR2240629}.
The first work in this context, due to Lovadina and Stenberg \cite{MR2240629}, is based on the a posteriori error analysis of an equivalent method, which involves the postprocessed approximation proposed in \cite{stenberg_1991}. In contrast, \cite{MR2914423} and \cite{MR2365826} plug the postprocessed approximation into the estimator without considering any equivalent scheme, while \cite{MR3082572} used the difference between the primal discrete solution and its postprocessing to drive $ p $-adaptivity. Moreover, in \cite{MR2914423}, the authors provide estimates based on a slight modification of the postprocessing scheme \cite{stenberg_1991} in the context of HDG methods. On the other hand, \cite{MR2365826} generalizes the work in \cite{MR2240629} in the following sense: firstly, it can use any piecewise polynomial approximation of the primal variable when computing the residual; and secondly, it can also be applied to stabilized mixed methods.

Among other methods that include a posteriori error estimators employing postprocessing techniques, we highlight the work in mixed FEM \cite{MR2403589}, hybrid mixed FEM \cite{MR3167450}, and HDG methods for: Kirchhoff-Love plates \cite{MR3902886}; Stokes/Brinkman equations \cite{araya_solano_vega_2019a}; and Oseen equations \cite{araya_solano_vega_2019b}, to name a few of them.

In the context of FEMs based on residual minimization using the residual representative to drive adaptive refinements, we emphasize Least-Squares Finite Element Methods (LS-FEM)~\cite{MR2765484,MR1615154}, DPG methods~\cite{CarDemGop_SINUM,DPG3}, 
and the recent work on Adaptive Stabilized Finite Element Method (AS-FEM)~\cite{calo2020adaptive}.

\subsection{Outline of the paper}
The remainder of the paper is organized as follows. In Section \ref{sec:settings}, we introduce notation, the model problem, and a mixed finite element discretization to solve it. We also introduce the Stenberg's postprocessing scheme for the mixed FEM. In Section \ref{sec:main_results}, we present the postprocessing based on local residual minimization, and our main results, related with a posteriori error estimates. We report numerical experiments that illustrate our key findings in Section \ref{sec:experiments}, and draw our conclusions in Section \ref{sec:conclusions}. We also provide an appendix with some auxiliary results needed for the a posteriori theory, namely: the existence of a Fortin operator, and intermediate estimates needed for the proof of the reliability result.

\section{Theoretical background}\label{sec:settings}

\subsection{Notation}\label{notation}

Let $ \Omega\subset\mathbb{R}^d $ ($ d=2,3 $) be a polyhedral Lipschitz domain.  At the continuous level,  we denote by $\CV \eq L^2(\Omega) $ the space of scalar real-valued square-integrable functions defined over $ \Omega $, while its vector-valued version will be denoted by $\BCV \eq \BL^2(\Omega) $. Additionally, $\BCH \eq H(\ddiv, \Omega)$ will denote the space of functions belonging to  $\BCV$ such that its divergence belongs to $\CV$. Besides, $ H^s(\Omega)\eq W^{s,2}(\Omega)$ will denote the usual Sobolev space of square-summable functions up to the $s^{\hbox{\tiny th}}$-derivative order (see~\cite{adams_fournier_2003a}), and $H^{1/2}(\partial\Omega)$ will denote the space of \emph{traces} of functions in $H^1(\Omega)$. 
Moreover,  for any $D\subset\overline{\Omega}$, we will use $(\cdot,\cdot)_D$ and $\|\cdot\|_D$ to denote the standard $L^2(D)$-inner product and inherited norm (respectively), for either scalar and vector-valued functions.

We consider a family of meshes $\{\CT_h\}_{h>0}$, each of them partitioning $ \Omega $ into non-overlapping simplexes $ K $, where $ h $ denotes the global mesh size defined as $ h\eq\max_{K\in\CT_h}h_K $. We assume that each mesh is conforming in the sense of \cite{ciarlet} and shape-regular\footnote{That is, there exists a constant $ \sigma>0 $, such that $ h_K/\rho_K\leq\sigma $ for all $ K\in\CT_h $ and for all $ h>0 $, where $ h_K $ denotes the diameter of $ K $, and $ \rho_K $ denotes the diameter of the largest ball inside $ K $.}. Let $ h_F $ be the diameter of a face/edge $ F $ of a simplex $ K $. We further denote by $ \CF_h^{\rm e} $ the set of exterior faces/edges lying on the boundary $ \partial\Omega $; by $ \CF_h^{\rm i} $ the remaining (interior) faces/edges; by $ \CF_h\eq\CF_h^{\rm i}\cup\CF_h^{\rm e} $ the set of all faces/edges; and for $ K\in\CT_h $, we define the sets $ \CF_K\eq\{F\in\CF_K:F\subset\partial K\} $, $ \CF_K^{\rm i}\eq\CF_K\cap\CF_h^{\rm i} $, and $ \CF_K^{\rm e}\eq\CF_K\cap\CF_h^{\rm e} $.

For a scalar-valued function $ w:\overline{\Omega}\to\mathbb{R} $, we define the jump of $ w $ across $ F\in\CF_h $ as
\begin{align*}
	\jmp{w}_F\eq\begin{cases}
		w^+|_F-w^-|_F&\text{if }F=\partial K^+\cap\partial K^-\quad\text{(with $ w^\pm\eq w|_{K^\pm} $)},\\
		w|_F&\text{if }F\in\CF_h^{\rm e},
	\end{cases}
\end{align*}
where $ K^+ $ and $ K^- $ are such that the normal to $ F $ points from $ K^+ $ to $ K^- $. We will omit the subscript $ F $ in $ \jmp{\cdot}_F $ if the jump is inside an $ L^2(F) $ norm.

For each $ K\in\CT_h $, consider the following subspace of zero-mean functions in $ H^1(K) $
\begin{equation}\label{zeromeanH1}
	H_*^1(K)\eq\left\{v_K\in L^2(K): \grad v_K \in \left(L^2(K)\right)^d \hbox{ and } (v_K,1)_K=0\right\}.
\end{equation}
The space \eqref{zeromeanH1} is endowed with the norm $\|\grad(\cdot)\|_K$ and inner-product $(\grad u_K,\grad v_K\big)_K$, for all $u_K,v_K\in H_*^1(K)$. Accordingly, we consider the following dual norm for the corresponding dual space 
$(H_*^1(K))'$ of bounded linear functionals over $H_*^1(K)$
$$
\|\cdot\|_{(H_*^1(K))'}\eq\sup_{v_K\in H_*^1(K)}{\left<\,\cdot\,,v_K\right>\over \|\grad v_K\|_K}\,,
$$ 
where $\left<\,\cdot\,,\,\cdot\,\right>$ denotes the associated duality pairing.

For the broken space $ H^1(\CT_h)\eq\{w\in L^2(\Omega):w|_K\in H^1(K)\quad\forall K\in\CT_h\} $, we introduce the global $ H^1(\CT_h) $-seminorm $$ \|\grad(\cdot)\|_{\CT_h}^2\eq\sum_{K\in\CT_h}\|\grad(\cdot)\|_K^2\,, $$
and the norm $ \|\cdot\|_{1,h}^2\eq\sum_{K\in\CT_h}|\cdot|_{1,K,h}^2 $, where $$ |\cdot|_{1,K,h}^2\eq\|\grad(\cdot)\|_K^2+\frac{1}{2}\sum_{F\in\CF_K^{\rm i}}h_F^{-1}\big\|\!\jmp{\cdot}\big\|_{F}^2+\sum_{F\in\CF_K^{\rm e}}h_F^{-1}\big\|\cdot\big\|_{F}^2. $$

For $ \bp\in\BCH $ such that $\bp\cdot\bn\in L^2(\partial K)$ (for all $ K\in\CT_h $), we consider the norm
$$\|\bp\|_{0,h}^2\eq\sum_{K\in\CT_h}\|\bp\|_{0,K,h}^2,\qquad\text{where }
\|\bp\|_{0,K,h}^2\eq\|\bp\|_K^2+h_K\|\bp\cdot\bn\|_{\partial K}^2.
$$

At the discrete level, for a given integer $r \geq 0$, let $\CV_K^r$ be the space of polynomials functions over $K$ of total degree no greater than $r$, and let $\BCV_K^r\eq\left(\CV_K^r\right)^d$ denote its vector-valued version. We will also consider the zero-mean space $ \CV_{*,K}^r \eq \CV_K^r\cap H_*^1(K)$, together with the global spaces
\begin{align}\label{eq:L2norm_extended}
	\CV_h^r&\eq\{v_h\in L^2(\Omega):v_h|_K\in\CV_K^r\quad\forall K\in\CT_h\} \,,\\
	\CV_{*,h}^r&\eq\{v_h\in L^2(\Omega):v_h|_K\in\CV_{*,K}^r\quad\forall K\in\CT_h\} \, .
\end{align}
Moreover, let $\BCH_h^r \eq \BCV_h^r\cap\BCH$, where $ \BCV_h^r\eq\{\bv_h\in\BCV:\bv_h|_K\in\BCV_K^r\quad\forall K\in\CT_h\} $.

Additionally, for $ \bp\in\BCH $ such that $\bp\cdot\bn\in L^2(\partial K)$ (for all $ K\in\CT_h $), we will also consider the norm
$$\|\bp\|_{*,h}^2\eq\sum_{K\in\CT_h}\|\bp\|_{*,K,h}^2,\qquad\text{where }
\|\bp\|_{*,K,h}^2\eq\|\bp\|_{*,K}^2+h_K\|\bp\cdot\bn\|_{\partial K}^2,
$$
where $ \|\cdot\|_{*,K} $ is a discrete dual norm defined by
\begin{align}\label{dual_discreta}
	\|\cdot\|_{*,K}\eq\sup_{v_K\in\CV_{*,K}^{p+2}}\frac{(\ \cdot\ ,\grad v_K)_K}{\|\grad v_K\|_K}.
\end{align}

For each $ K\in\CT_h $, we denote by $ Q_K^r:L^2(K)\to\CV_K^r $ the $ L^2 $-projection from $ L^2(K) $ onto $ \CV_K^r $, and $ Q_h^r:L^2(\Omega)\to \CV_h^r $ its global counterpart, namely, $ (Q_h^r v)|_K=Q_K^r(v|_K) $. In particular, we set $ Q_h\eq Q_h^0 $ and $ Q_K\eq Q_K^0 $ for each $ K\in\CT_h $, and we recall that the action of the operator $ I-Q_K $ provides zero-mean functions on $ K $.	

In the remainder of this document, if $ A, B \geq 0 $, we employ the notation $ A\lesssim B $ if there exists a positive constant $ C $, independent of $ h $, such that $ A \leq CB $.

\subsection{Model problem}

To fix ideas, given $f\in\CV$ and $ u_D\in H^{1/2}(\partial\Omega)$, we want to approximate the solution $u\in H^1(\Omega)$ and the flux $\bq\eq-\grad u\in \BCV$ of the model problem
\begin{subequations}\label{eq:model_problem}
	\begin{align}
		-\Delta u&=f\ \ \quad\text{in }\Omega,\\
		u&=u_D\quad\text{on }\partial\Omega.
	\end{align}
\end{subequations}

In order to approximate $u$ and $\bq$ as primal variables, we write our model problem as the system
\begin{subequations}\label{eq:model_system}
	\begin{align}\label{eq:constitutive}
		\bq+\grad u&={\bf 0}\ \ \quad\text{in }\Omega,\\
		\div\bq&=f\ \ \quad\text{in }\Omega,\\
		u&=u_D\quad\text{on }\partial\Omega\,,
	\end{align}
\end{subequations}
which admits the following variational formulation: 
\begin{subequations}\label{eq:continuous_formulation}
	\begin{align}
		\hbox{Find $ (u,\bq)\in\CV\times\BCH $ such that}\notag\\
		(\bq,\bp)_{\CT_h}-(\div\bp,u)_{\CT_h} = & -\langle u_D,\bp\cdot\bn\rangle_{\partial\Omega}
		&&\hspace{-1.5cm} \forall\bp\in\BCH,\label{eq:continuous_formulation_a}\\
		(\div\bq,v)_{\CT_h}= & \,(f,v)_{\CT_h}
		&&\hspace{-1.5cm}\forall v\in\CV.\label{eq:cont_form_div}
	\end{align}
\end{subequations}

Formulation~\eqref{eq:continuous_formulation} is well-posed by means of Babu\v{s}ka-Brezzi theory (see, e.g.,~\cite[Theorem~49.3]{ern_guermond_2021b}).

\subsection{A mixed FEM}

As an illustrative example of a mixed FEM discretizing~\eqref{eq:continuous_formulation}, we present the $d$-dimensional Brezzi-Douglas-Marini (BDM) finite element family~\cite{BDDF,BDM}, together with its a priori error estimates. The well-posedness of this discrete formulation relies on standard Babu\v{s}ka-Brezzi theory (see, e.g.,~\cite[Section~50.1.1]{ern_guermond_2021b}).

For a fixed integer $p\ge1$, we consider the following conforming and fully-discrete version of~\eqref{eq:continuous_formulation}: 
\begin{subequations}\label{eq:discrete_formulation}
	\begin{align}
		\qquad\hbox{Find $ (u_h,\bq_h)\in\CV_h^{p-1}\times\BCH_h^p $ such that}\notag\\
		(\bq_h,\bp_h)_{\CT_h}-(\div\bp_h,u_h)_{\CT_h}= & -\langle u_D,\bp_h\cdot\bn\rangle_{\partial\Omega}
		&&\forall\bp_h\in\BCH_h^p\subset\BCH,\\
		(\div\bq_h,v_h)_{\CT_h}= &\, (f,v_h)_{\CT_h}
		&&\forall v_h\in\CV_h^{p-1}\subset\CV.\label{eq:disc_form_div}
	\end{align}
\end{subequations}

Now, we recall the a priori error estimates for~\eqref{eq:discrete_formulation}; cf.~\cite[Theorem~2.1]{stenberg_1991}.

\begin{theorem}\label{estimates_mixed}
	Let $ (u,\bq)\in\CV\times\BCH $ be the solution of \eqref{eq:continuous_formulation} and $ (u_h,\bq_h)\in\CV_h^{p-1}\times\BCH_h^p $ be the solution of \eqref{eq:discrete_formulation}. Let us suppose that $ u $ lies in $ H^s(\Omega) $ with $ s>3/2$. Then, we have
	\begin{subequations}
		\begin{alignat}{2}
			\|\bq-\bq_h\|_{0,h} & \lesssim h^r|\bq|_{s,\Omega}\,, &&\quad \hbox{with } r=\min\{s-1,p+1\},\label{apriori_q}\\
			\|u-u_h\|_{\Omega} & \lesssim h^l(|\bq|_{l,\Omega}+|u|_{l,\Omega})\,, &&  \quad\hbox{with } l=\min\{s-1,p\},
		\end{alignat}
	\end{subequations}
	where $ |\cdot|_{s,\Omega} $ is the standard $ H^s(\Omega) $-seminorm \cite{adams_fournier_2003a}.	Moreover, if $ \Omega $ is convex, then we have 
	\begin{alignat}{2}
		\|u-u_h\|_{\Omega}&\lesssim h^l(|\bq|_{l-1,\Omega}+|u|_{l,\Omega}),\quad with\  l=\min\{s,p\},
		\intertext{and}
		\|u_h-Q_h^{p-1}u\|_{\Omega}&\lesssim\label{error_estimate}
		\begin{cases}
			h^{r+1}|\bq|_{r,\Omega},\quad with\  r=\min\{s-1,p+1\}&\quad\text{for }p\geq 2,\\
			h^2|\bq|_{2,\Omega}&\quad\text{for }p=1.
		\end{cases}
	\end{alignat}
\end{theorem}
If $ f\in\CV_h^{p-1} $, then the first estimate in \eqref{error_estimate} also holds for $ p=1 $.
\begin{remark}
	Since $ \|\bq-\bq_h\|_{*,K}\leq\|\bq-\bq_h\|_K $, for all $ K\in\CT_h $, we have the same a priori error estimate \eqref{apriori_q} for $ \|\bq-\bq_h\|_{*,h} $\,.
\end{remark}

\subsection{A superconvergent postprocessing}
\!Let $ (u_h,\bq_h) $ be the solution of~\eqref{eq:discrete_formulation}. We present a slight modification of the postprocessing scheme introduced by Stenberg in~\cite{stenberg_1991}. We look for a discrete solution $ \widetilde{u}_{h} \in\CV_{h}^{p+1} $ satisfying the following two equations on each element $ K\in\CT_h$:
\begin{subequations}\label{eq:postproc}
	\begin{align}
		\label{eq:postproc_a}
		(\nabla\widetilde{u}_{h},\nabla v_K)_K & = -(\bq_h,\grad v_K)_K
		&&\hspace{-2cm}\forall \, v_K\in\CV_{*,K}^{p+1}\,,\\
		\label{eq:postproc_b}
		(\widetilde{u}_{h},1)_K  & = (u_h,1)_K\,. &&
	\end{align}
\end{subequations}

\begin{remark}
	Since $ \widetilde{u}_h $ is computed locally, the computational effort required to obtain $ \widetilde{u}_h $ is negligible compared to the computational effort needed to solve the global problem~\eqref{eq:discrete_formulation}.
\end{remark}

For the postprocessed solution $ \widetilde{u}_h $, the next Theorem shows an a priori error estimate (cf.~\cite[Theorem 2.2]{stenberg_1991}). We recall its proof for the sake of completeness and to emphasize its dependence on a priori error estimates associated with the mixed method \eqref{eq:discrete_formulation}.

\begin{theorem}\label{thm:post_apriori}
	Let $ (u,\bq)\in\CV\times\BCH $ be the solution of \eqref{eq:continuous_formulation} and $ \widetilde{u}_h $ be defined by \eqref{eq:postproc}. Then, we have
	\begin{align}\label{apriori_pp}
		\|u-\widetilde{u}_h\|_K
		&\leq
		\frac{h_K}{\pi}\!\left(\|\bq-\bq_h\|_K+\left\|\grad\big(u-Q_K^{p+1}u|_K\big)\right\|_K\,\right)\\
		&\quad\ +\left\|u_h-Q_K^{p-1}u|_K\,\right\|_K
		+\left\|u-Q_K^{p+1}u|_K\,\right\|_K,\nonumber
	\end{align}
	for all $ K\in\CT_h $.
\end{theorem}

\begin{proof}
	Let us define $ v\in\CV_{*,h}^{p+1} $ through $ v|_K=(I-Q_K)\big(Q_K^{p+1}u|_K-\widetilde{u}_h|_K\big) $, for each $ K\in\CT_h $. Thanks to \eqref{eq:postproc_a}, and the fact that $ \grad u + \bq=0 $, we write
	\begin{align*}
		\|\grad v\|_K^2&=\left(\grad(I-Q_K)\big(Q_K^{p+1}u|_K-\widetilde{u}_h|_K\big),\grad v\right)_{\!K}\\
		&=\left(\grad\big(Q_K^{p+1}u|_K-\widetilde{u}_h\big),\grad v\right)_{\!K}\\
		&=\left(\grad Q_K^{p+1}u|_K,\grad v\right)_{\!K}+(\bq_h,\grad v)_K\\
		&=-\left(\grad\big(u-Q_K^{p+1}u|_K\big),\grad v\right)_{\!K}-(\bq-\bq_h,\grad v)_K\,.
	\end{align*}
	Thus, by using the Cauchy-Schwarz inequality, we obtain
	\begin{equation}\label{eq:bound_v}
		\|\grad v\|_K\leq\left\|\grad\big(u-Q_K^{p+1}u|_K\big)\right\|_K+\|\bq-\bq_h\|_K\,.
	\end{equation}
	Since $ v|_K\in\CV_{*,K}^{p+1} $, thanks to \eqref{eq:bound_v} and Poincar\'e inequality \cite[Eq.~(4.3)]{payne_weinberger_1960}, we have
	\begin{align}\label{eq:bound_v_final}
		\|v\|_K&\leq\frac{h_K}{\pi}\|\grad v\|_K
		\leq\frac{h_K}{\pi}\!\left(\left\|\grad\big(u-Q_K^{p+1}u|_K\big)\right\|_K
		+\|\bq-\bq_h\|_K\right).
	\end{align}
	On another hand, observe that
	\begin{align*}
		\int_K Q_K Q_K^{p+1}u|_K=\int_K Q_K^{p+1}u|_K=\int_K u=\int_K Q_K^{p-1}u|_K=\int_K Q_K Q_K^{p-1}u|_K\,.
	\end{align*}
	Therefore, we deduce that $ Q_K Q_K^{p+1}u|_K=Q_K Q_K^{p-1}u|_K $. Moreover, recalling \eqref{eq:postproc_b} and the boundedness of $ Q_K $, we conclude that
	\begin{align}\label{eq:proj_term}
		\big\|Q_K\big(Q_K^{p+1}u|_K-\widetilde{u}_h|_K\big)\!\big\|_K\!=\!\big\|Q_K\big(Q_K^{p-1}u|_K-u_h|_K\big)\!\big\|_K\!\leq\!\big\|Q_K^{p-1}u|_K-u_h\big\|_K\,.
	\end{align}
	Hence, \eqref{apriori_pp} follows from \eqref{eq:bound_v_final}, \eqref{eq:proj_term}, and the splitting
	\begin{align*}
		(u-\widetilde{u}_h)|_K=v|_K+Q_K\big(Q_K^{p+1}u|_K-\widetilde{u}_h|_K\big)+\big(u|_K-Q_K^{p+1}u|_K\big),
	\end{align*}
	for all $ K\in\CT_h $.
\end{proof}

The extension of this analysis to the other mixed methods of \cite{BDDF,BDM,nedelec,RT,stenberg_2010} is straightforward.

\begin{remark}{(Optimal convergence})\label{thm:apriori_aux}
	Thanks to the approximation properties of $ Q_h^{p+1} $ and the estimates of Theorems \ref{estimates_mixed} and \ref{thm:post_apriori}, if $ u\in H^s(\Omega) $, with $ s>3/2 $, and $ \Omega $ is convex, then we have
	\begin{equation}\label{apriori_minres_aux}
		\|u-\widetilde{u}_h\|_{\CT_h}\lesssim
		\begin{cases}
			h^{r+1}(|\bq|_{r,\Omega}+|u|_{r+1,\Omega}),\quad r=\min\{s-1,p+1\},&\text{ for }p\geq 2,\\
			h^2(|\bq|_{2,\Omega}+|u|_{2,\Omega}),&\text{ for }p=1.
		\end{cases}
	\end{equation}
	Moreover, if $ f\in\CV_h^{p-1} $, then the first estimate in \eqref{apriori_minres_aux} also holds for $ p=1 $.
\end{remark}

\section{Main results}\label{sec:main_results}

In this section we introduce our postprocessing scheme and state its relation with the postprocessing scheme proposed by Stenberg~\cite{stenberg_1991}. Next, we present our main results, concerning a posteriori error estimates.

\subsection{Postprocessing based on local residual minimization}

We present the main idea behind our postprocessing procedure, which considers the local postprocessing scheme \eqref{eq:postproc} as a starting point.

In our case, by keeping the same approximation space $\CV_h^{p+1}$, we will follow a residual minimization approach. 
Consequently, the new postprocessed variable $\nu_h\in \CV_h^{p+1}$ will be defined locally through the following set of local residual minimization problems:
\begin{subequations}\label{eq:resmin}
	\begin{align}
		\notag  &\hspace{-1.75cm}\text{Find $\nu_h\in \CV_h^{p+1}$  such that, for each $K\in\CT_h$,}\\
		\label{eq:resmin_a}
		& \displaystyle\nu_{h}|_K=\argmin_{w_K\in \CV_K^{p+1}}\frac{1}{2}\|\bq_h+\grad w_K\|_{*,K}^2
		\,, &&\\
		( \nu_h&,1)_K = (u_h,1)_K\,,&&
	\end{align}
\end{subequations}
where $ \|\cdot\|_{*,K} $ is the discret dual norm defined in~\eqref{dual_discreta}.

For each $K\in\CT_h$, let $ \nu_K\eq\nu_h|_K $. We emphasize that solving \eqref{eq:resmin} is equivalent to solve the local mixed linear system (see \cite{MR2916380}):
\begin{subequations}\label{eq:saddlepoint}
	\begin{align}
		\notag	\qquad\text{Find $ (\varepsilon_K,\nu_{K})\in \CV_{*,K}^{p+2}\times \CV_K^{p+1} $ such that}\\
		(\grad\varepsilon_K,\grad v_K)_K+(\grad\nu_{K},\grad v_K)_K= & \, -(\bq_h,\grad v_K)_K
		&& 
		\qquad\forall v_K\in \CV_{*,K}^{p+2},\label{eq:saddlepoint_a}\\
		(\grad w_K,\grad\varepsilon_K)_K\hspace{2.625cm}= & \, 0
		&& \qquad\forall w_K\!\in \CV_{*,K}^{p+1}\,,\label{eq:saddlepoint_b}\\
		( \nu_h,1)_K = & \, (u_h,1)_K.&&
	\end{align}
\end{subequations}

We also introduce $ \varepsilon_h\in\CV_{*,h}^{p+2} $ as the global version of $ \varepsilon_K $, that is, $ \varepsilon_h|_K\eq\varepsilon_K $ for each $ K\in\CT_h $.
Next, we establish how our postprocessed solution $ \nu_h $ relates to the Stenberg postprocess. In particular, both postprocessed solutions are the same. However, they have a different nature: while the first is constructed through a residual minimization approach, the second one solves directly an elliptic problem.

\begin{proposition}\label{prop:}
	Let $\widetilde u_h\in \CV_h^{p+1}$ be the solution of the postprocessing problem~\eqref{eq:postproc}, and let $\nu_h\in \CV_h^{p+1}$ be the solution of the residual minimization postprocessing~\eqref{eq:resmin}. Then, $ \nu_h = \widetilde u_h $.
	That is, both postprocessed solutions are the same.
\end{proposition}

\begin{proof}
	Since $ \CV_{*,K}^{p+1}\subset\CV_{*,K}^{p+2} $, for each $ K\in\CT_h $, we can test \eqref{eq:saddlepoint_a} with $ v_K\in\CV_{*,K}^{p+1} $. Moreover, thanks to \eqref{eq:saddlepoint_b}, the first term in \eqref{eq:saddlepoint_a} vanishes, so we have that $ \nu_K $ solves \eqref{eq:postproc}. By uniqueness of the solution of \eqref{eq:postproc}, we conclude that $ \nu_h = \widetilde u_h $.
\end{proof}

Given the fact that both postprocessed solutions are equal, hereafter we denote them uniquely by $\nu_h$.

\begin{corollary}[a priori error analysis]\label{corol:apriori_pp}
	The postprocessed solution $\nu_h\in~\CV_h^{p+1}$ also satisfies the a priori error estimate~\eqref{apriori_pp} given in Theorem~\ref{thm:post_apriori}.
\end{corollary}

\begin{remark}[General test spaces]
	We could have considered $\CV_{*,K}\subset H_*^1(K)$ as an arbitrary finite-dimensional test space containing $ \CV_{*,K}^{p+1} $ in \eqref{eq:saddlepoint_a}. In such a case, Proposition \ref{prop:} and Corollary \ref{corol:apriori_pp} still hold. However, we prefer $ \CV_{*,K}^{p+2} $ since higher-order polynomials will not increase the order of convergence.
	
	On the other hand, the local residual representative $\varepsilon_K\in(\CV_{*,K}^{p+1})^\perp\subset \CV_{*,K}$ is nothing but the Riesz representative of the local residual $-(\bq_h+\grad\nu_ h,\grad(\cdot))_K\in(\CV_{*,K})'$. 
	In general, $\varepsilon_K\neq 0$ whenever $\CV_{*,K}$ strictly\footnote{Notice that if $ \CV_{*,K}=\CV_{*,K}^{p+1} $, then $ \varepsilon_K=0 $, and therefore \eqref{eq:saddlepoint} is exactly the same as \eqref{eq:postproc}.} contains $\CV_{*,K}^{p+1}$. Hence, $\|\grad\varepsilon_K\|_K$ can be used as a local (a posteriori) error estimator to drive adaptivity.
\end{remark}

\subsection{A posteriori error analysis}
In this section, we introduce and analyze a posteriori error estimators based on the local residual representative $ \varepsilon_K $ provided by~\eqref{eq:saddlepoint}. In contrast to DPG methods, which directly use the (built-in) residual representative as an (a posteriori) error estimator, we further propose an improved error estimator using $ \varepsilon_K $ as a starting point (see \S\ref{sec:improved_estimator}).

\subsubsection{A built-in a posteriori error estimator}

In this section, we study the a posteriori error estimator
$
\widetilde{\eta}\eq\left(\sum_{K\in\CT_h}\widetilde{\eta}_K^2\right)^{1/2}$, with $\widetilde{\eta}_K\eq \|\grad\varepsilon_K\|_K\,,$
where $\varepsilon_K\in \CV_{*,K}^{p+2}$ is the first component of the solution of problem~\eqref{eq:saddlepoint}. We have realized that the associated reliability estimate allows to control the error $ \bq-\bq_h $, only in a norm weaker than the $ L^2(\Omega) $-norm. As in \cite{MR2240629}, full control of the $ L^2(\Omega) $-norm requires additional terms quantifying the mismatch associated with the discrete version of \eqref{eq:constitutive} (namely, $ \|\bq_h+\grad\nu_h\|_K $), together with interelement jumps of the postprocessed solution $ \nu_h $ (see \S\ref{sec:improved_estimator} for the details).

Before stating the reliability estimate, we want to highlight a detail regarding the norm that we are using to measure $ \bq-\bq_h $. In \cite{{stenberg_1991}}, the author needed to state an a priori error estimate for $ \bq $ in the mesh dependent norm $ \|\cdot\|_{0,h} $ given in \eqref{eq:L2norm_extended}. Even though we do not require the boundary part of $ \|\cdot\|_{0,h} $ to obtain the a priori estimate \eqref{apriori_minres_aux}, we can use the norm $ \|\cdot\|_{0,h} $ in the a posteriori error analysis to get a better measure of $ \bq-\bq_h $. However, the latter requires the following (Fortin) assumption:
\begin{assumption}[Fortin operator]\label{fortin_assumpt}
	Let $ \gamma_K:H^1(K)\to L^2(\partial K) $ denote the continuous Dirichlet trace operator. There exists an operator $ \Pi_{\partial K}:\gamma_K(H_*^1(K))\to \gamma_K(\CV_{*,K}^{p+2}) $ and a constant $ C_{\Pi}>0 $, independent of $ h $, such that the following conditions are satisfied:
	\begin{subequations}\label{Fortin}
		\begin{align}
			&\|\Pi_{\partial K} v_K\|_{\partial K}\leq C_{\Pi}\|v_K\|_{\partial K}\qquad
			\forall v_K\in\gamma_K(H_*^1(K)),\\
			&(\bp_h\cdot\bn,v_K-\Pi_{\partial K} v_K)_{\partial K}=0\quad\ \!\forall\bp_h\in\BCH_h^p,\ \forall v_K\in \gamma_K(H_*^1(K)).\label{Fortin_orthog}
		\end{align}
	\end{subequations}
\end{assumption}
\begin{remark}
	In Appendix \ref{appA}, we verify Assumption \ref{fortin_assumpt} for the case $ p=1 $.
\end{remark}
To state the last (saturation) assumption needed to prove the reliability estimate, we introduce the following auxiliary problem:
\begin{subequations}\label{eq:aux_problem}
	\begin{align}
		\notag	&\hspace{-2cm}\text{Find $ \theta_h\in\CV_h^{p+2} $ such that, for each $ K\in\CT_h $,}\\
		&(\grad\theta_h,\grad v_K)_K=  \, -(\bq_K,\grad v_K)_K
		\qquad\forall v_K\in \CV_{*,K}^{p+2},\\
		&\hspace{0.95cm}( \theta_h,1)_K =  \, (u_h,1)_K.
	\end{align}
\end{subequations}
Now, we introduce the last asssumption:
\begin{assumption}[Saturation]\label{eq:saturation}
	Let $ (\varepsilon_h,\nu_h)\in \CV_{*,h}^{p+2}\times \CV_h^{p+1} $ solve \eqref{eq:saddlepoint} and $ \theta_h\in \CV^{p+2}_{*,h} $ solve \eqref{eq:aux_problem}. There exists a real number $ \delta\in[0,1) $, uniform with respect to $ h $, such that $$ \|\grad\!\left(u-\theta_h\right)\!\|_{\CT_h}\leq\delta\|\grad\!\left(u-\nu_h\right)\!\|_{\CT_h}. $$
\end{assumption}

\begin{theorem}[Reliability based on the residual representative]\label{rel}
	Let $(u,\bq) \in \CV \times \BCH$ be the solution of the continuous problem~\eqref{eq:model_problem}; let $(u_h,\bq_h) \in \CV_h^{p-1} \times \BCH_h^p$ be the solution of the discrete problem~\eqref{eq:discrete_formulation}; and let $(\varepsilon_h,\nu_h)\in \CV_{*,h}^{p+2}\times \CV_h^{p+1} $ be the solution of \eqref{eq:saddlepoint}. 
	If Assumptions \ref{fortin_assumpt} and \ref{eq:saturation} are satisfied, then the following estimate holds:
	\begin{align}\label{eq:reliability}
		\|\grad(u-\nu_h)\|_{\CT_h}+\|\bq-\bq_h\|_{*,h}\lesssim\widetilde{\eta}+{\rm osc}(\bq),
	\end{align}
	where $ {\rm osc}(\bq)\eq\left(\sum_{K\in\CT_h}{\rm osc}_K(\bq)^2\right)^{1/2} $, and
	\begin{align}\label{osc}
		{\rm osc}_K(\bq)\eq
		h_K^{1/2}\sup_{v\in L^2(\partial K)\setminus\{0\}}\frac{\langle\bq\cdot\bn,v-\Pi_{\partial K}v\rangle_{\partial K}}{\|v\|_{\partial K}}.
	\end{align}
\end{theorem}
\begin{proof}
	The proof requires auxiliary results and has been shifted entirely to Appendix \ref{append:rel}.
\end{proof}

\begin{remark}[Upper bound for the oscillation term]
	As a consequence of the orthogonality property \eqref{Fortin_orthog} of the Fortin operator $ \Pi_{\partial K} $, the oscillation term $ {\rm osc}_K(\bq) $ is bounded above by the best-approximation error
	\begin{equation*}
		h_K^{1/2}\inf_{\bp_h\in\BCH_h^p}\|(\bq-\bp_h)\cdot\bn\|_{\partial K}.
	\end{equation*}
\end{remark}

\begin{remark}[PDE-independence]
	Notice that our a posteriori error estimates are independent of the underlying PDE. Enjoying such a property was one of the main reasons why we have used the discrete dual norm contained in $ \|\cdot\|_{*,h} $. In fact, if we had proceeded bounding $ \|\bq-\bq_h\|_{\Omega} $ as in \cite{MR2240629}, we would have found the quantity $ \div(\bq-\bq_h) $ (see \cite[Theorems 3.2 and 3.3]{MR2240629}), which may generate PDE-dependent additional terms in the error estimates. Let us be more precise. Consider the following PDE in primal and mixed forms 
	\begin{align*}
		(-\Delta+\CR)u=f\hspace{1cm}\Longleftrightarrow\hspace{1cm}\begin{cases}
			\hspace{0.55cm}\bq+\grad u=0\\
			\div\bq+\CR u=f
		\end{cases}\!\!\!\!,
	\end{align*}	
	where $ \CR $ is a non-diffusive differential operator. We have two possible scenarios: If $ \CR $ is the null operator (i.e., Poisson equation), then $ \div(\bq-\bq_h)=f-Q_h^{p-1}f $, which is a standard data oscillation term in a posteriori error estimation. However, if $ \CR $ is not the null operator (i.e., a more general PDE), then we have
	\begin{align*}
		\div(\bq-\bq_h)=f-\div\bq_h-\CR\nu_h-\CR(u-\nu_h),
	\end{align*}
	which clearly contains undesirable PDE-dependent additional terms to control.
	
	To properly bound $ \|\bq-\bq_h\|_{\Omega} $ without loosing PDE-independency, we will include additonal terms in an improved a posteriori estimator (see \S\ref{sec:improved_estimator}).
\end{remark}	

To close this section, we state and prove our local efficiency estimate. 

\begin{theorem}[Local efficiency]\label{eff}
	The following holds true:
	\begin{align}\label{eq:efficiency}
		\widetilde{\eta}_K&\leq \|\grad(u-\nu_h)\|_K+\|\bq-\bq_h\|_{*,K}\,,
	\end{align}
	for all $ K\in\CT_h$. 
\end{theorem}
\begin{proof}
	Thanks to equation~\eqref{eq:saddlepoint_a} and the fact that $ \bq+\grad u=0 $, for each $ K\in\CT_h $, we have
	\begin{align*}
		(\grad\varepsilon_K,\grad v_K)_K&=-(\grad\nu_K,\grad v_K)_K
		-(\bq_h,\grad v_K)_K\\
		&=(\grad(u-\nu_K),\grad v_K)_K+(\bq-\bq_h,\grad v_K)_K\,,
	\end{align*}
	for all $ v_K\in \CV_{*,K}^{p+2}$\,. Thus, we get
	\begin{align*}
		\widetilde{\eta}_K=\sup_{v_K\in \CV_{*,K}^{p+2}}\frac{(\grad\varepsilon_K ,\grad v_K)_K}{\|\grad v_K\|_K}\leq \|\grad(u-\nu_K)\|_K
		+\|\bq-\bq_h\|_{*,K},
	\end{align*}
	and \eqref{eq:efficiency} follows.
\end{proof}

\subsubsection{An improved a posteriori error estimator}\label{sec:improved_estimator}

We study here the behavior of an alternative to the estimator $ \widetilde{\eta}_K\eq\|\grad\varepsilon_K\|_K $, which let us deal with estimates related to $ \|\bq-\bq\|_K $ (instead of the discrete dual norm $ \|\bq-\bq_h\|_{*,K} $). Let us define the estimator $ \eta\eq\left(\sum_{K\in\CT_h}\eta_K^2\right)^{1/2} $, with
\begin{equation}\label{eq:eta}
	\eta_K^2\eq \widetilde{\eta}_K^2+\|\bq_h+\grad\nu_h\|_K^2+\frac{1}{2}\sum_{F\in\CF_K^{\rm i}}h_F^{-1}\|\!\jmp{\nu_h}\|_{F}^2+\sum_{F\in\CF_K^{\rm e}}h_F^{-1}\|u_D-\nu_h\|_{F}^2\,.
\end{equation}

To start, we present the reliability result for our improved estimator $ \eta $.

\begin{theorem}[Reliability, improved estimator]\label{rel_improved}
	Let $(u,\bq) \in \CV \times \BCH$ be the solution of the continuous problem~\eqref{eq:model_problem}; let $(u_h,\bq_h) \in \CV_h^{p-1} \times \BCH_h^p$ be the solution of the discrete problem~\eqref{eq:discrete_formulation}; and let $(\varepsilon_h,\nu_h)\in \CV_{*,h}^{p+2}\times \CV_h^{p+1} $ solve \eqref{eq:saddlepoint}. If Assumptions \ref{fortin_assumpt} and \ref{eq:saturation} are satisfied, then the following estimate holds:
	\begin{align}\label{eq:reliability_alt}
		\|u-\nu_h\|_{1,h}+\|\bq-\bq_h\|_{0,h}\lesssim\eta+{\rm osc}(\bq).
	\end{align}
	
\end{theorem}
\begin{proof}
	See Appendix \ref{append:rel_improved}.
\end{proof}

To conclude this section, based on Theorem \ref{eff}, we derive the following local efficiency estimate:

\begin{theorem}[Local efficiency, improved estimator]\label{eff_improved}
	For all $ K\in\CT_h$, the following estimate holds true:
	\begin{align}\label{eq:efficiency_alt}
		\eta_K&\leq |u-\nu_h|_{1,K,h}+\|\bq-\bq_h\|_K\,.
	\end{align} 
\end{theorem}

\begin{proof}
	Notice the identities
	\begin{align*}
		\big\|\!\jmp{\nu_h}\big\|_{F_{\rm i}}&=\big\|\!\jmp{u-\nu_h}\big\|_{F_{\rm i}}\qquad\text{and}\qquad\|u_D-\nu_h\|_{F_{\rm e}}=\|u-\nu_h\|_{F_{\rm e}}\,,
	\end{align*}
	for $ F_{\rm i}\in\CF_K^{\rm i} $ and $ F_{\rm e}\in\CF_K^{\rm e} $, respectively.
	Thus, \eqref{eq:efficiency_alt} follows from \eqref{eq:efficiency}, the triangle inequality, and the fact that $ \|\bq-\bq_h\|_{*,K}\leq\|\bq-\bq_h\|_K $, for all $ K\in\CT_h $\,.
\end{proof}

\section{Numerical validation}
\label{sec:experiments}
In this section, we present three numerical examples (in two dimensions) to illustrate the performance of our scheme and validate our theoretical findings. We take a BDM mixed finite element discretization \cite{BDDF,BDM} as the input for our scheme. We consider $ h $-adaptive iterative refinements combining D\"orfler marking together with a bisection-type refinement criterion \cite{bank1983some}.

\subsection{A smooth solution}
We consider the Poisson problem \eqref{eq:model_problem} over the unit square $ \Omega\eq(0,1)^2\subset\mathbb{R}^2 $. We set $ u_D(x_1,x_2)\eq 0 $ on $ \partial\Omega $ and $ f $ such that the exact solution is given by $ u(x_1,x_2)\eq x_1(1-x_1)\sin(\pi x_2) $ in $ \Omega $. We study the behavior of our scheme together with our two a posteriori error estimators. We analyze convergence rates and the effectivity index associated with $ \eta $ for uniformly refined meshes.

For different polynomial order $ p $, Figures \ref{fig:curves_smooth_p1}, \ref{fig:curves_smooth_p2}, and \ref{fig:curves_smooth_p3} show the convergence curves associated with each error measurement, namely, $ \|u-\nu_h\|_{1,h} $, $ \|\bq-\bq_h\|_h $, and $ \|\bq-\bq_h\|_{*,h} $. The first one exhibits superconvergence and the last two follow optimal convergence rates, as expected. In addition to the errors, we display the curves associated with both estimators, namely, $ \eta $ and $ \widetilde{\eta} $. Figure \ref{fig:curves_smooth_L2} compares the converge rates between the $ L^2 $-error associated with the mixed BDM scheme, and the $ L^2 $-error of the postprocessed solution.

\begin{figure}[h!]
	\centering
	\begin{subfigure}{0.49\linewidth}\vspace{0.1cm}
		\centering
		\includegraphics[width=\linewidth]{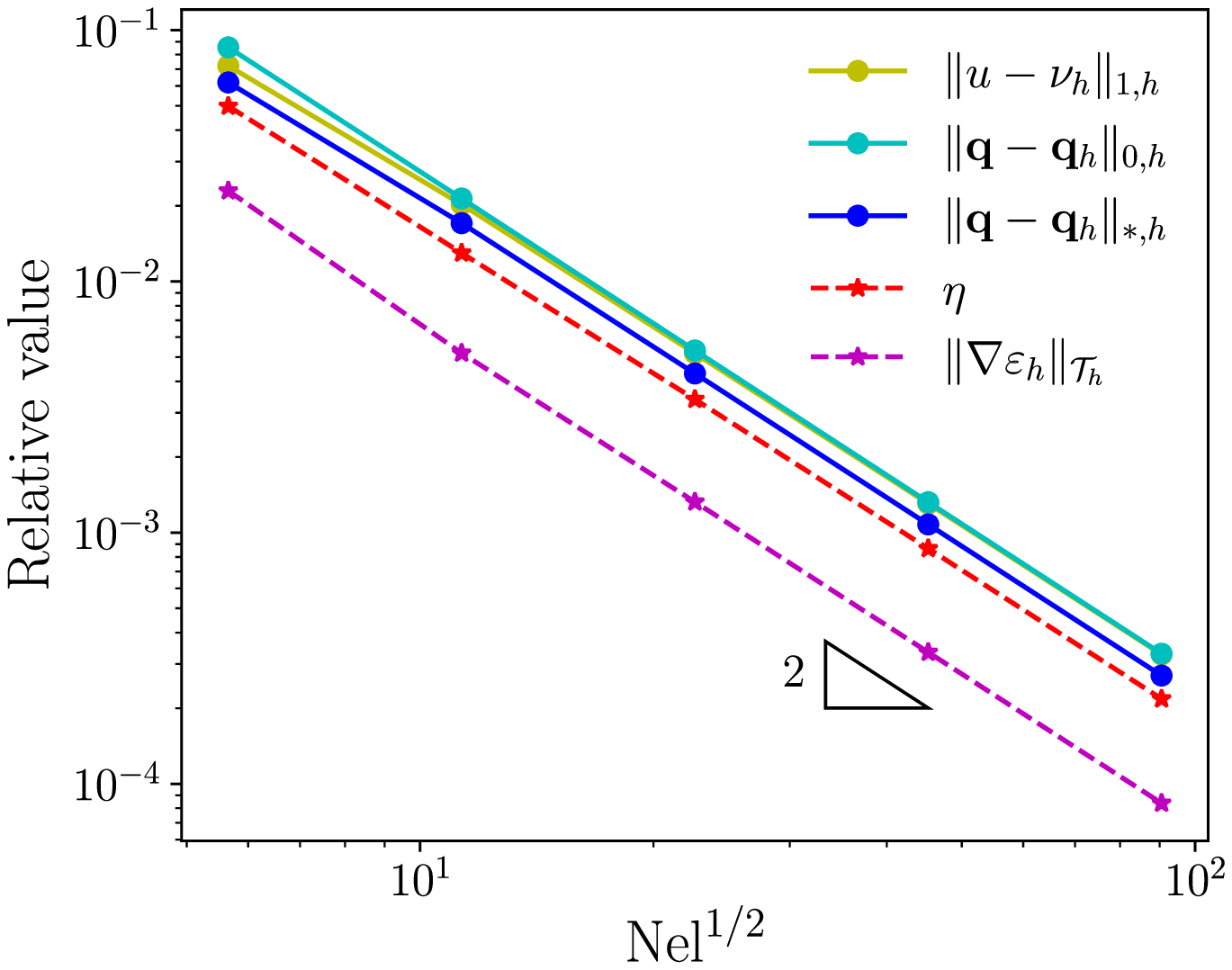}
		\caption{Error/estimator vs. $ {\rm Nel}^{1/2} $, $ p=1 $}
		\label{fig:curves_smooth_p1}
	\end{subfigure}
	\hfill
	\begin{subfigure}{0.49\linewidth}
		\centering
		\includegraphics[width=\linewidth]{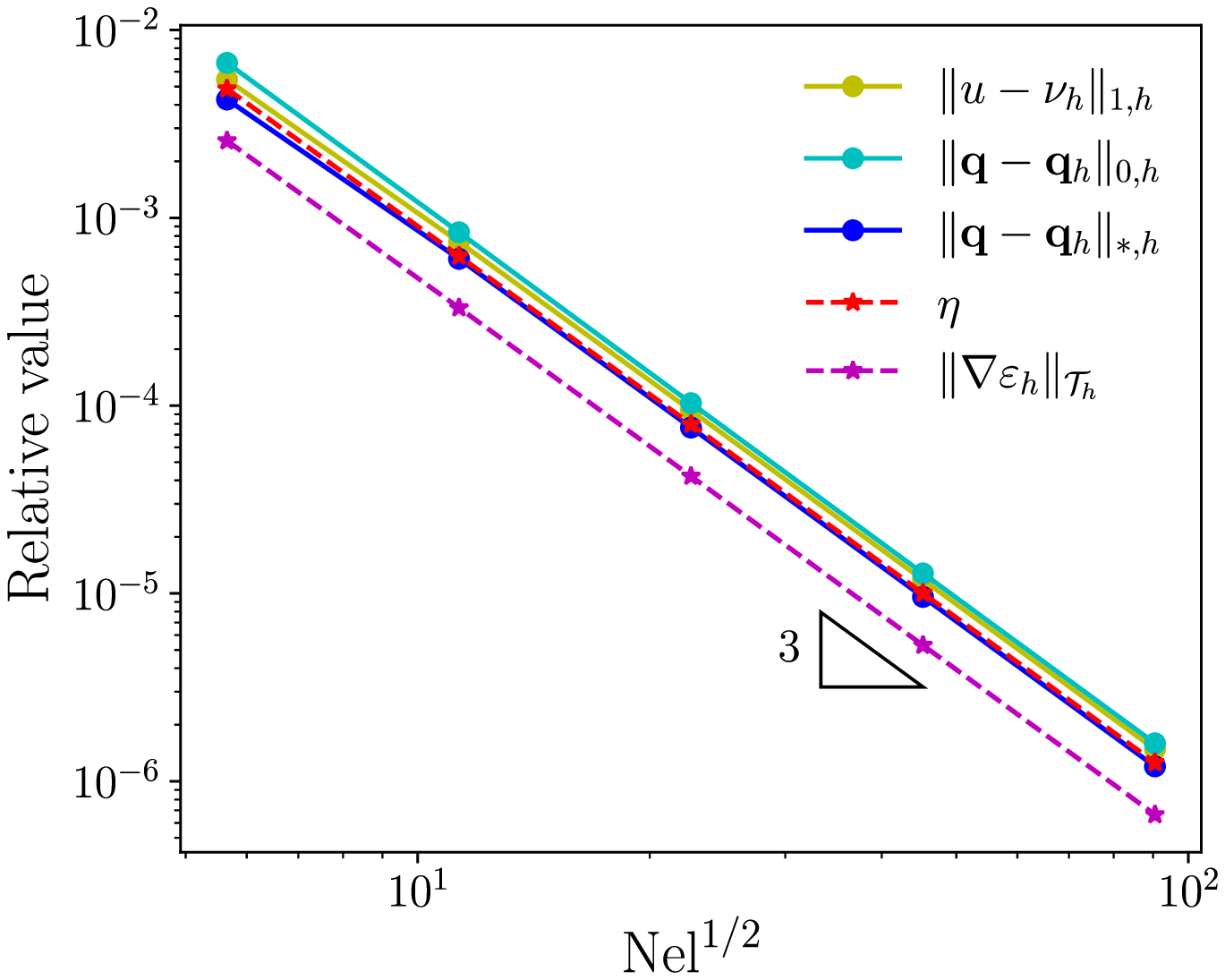}
		\caption{Error/estimator vs. $ {\rm Nel}^{1/2} $, $ p=2 $}
		\label{fig:curves_smooth_p2}
	\end{subfigure}\\
	\vspace{0.25cm}
	\begin{subfigure}{0.49\linewidth}
		\centering
		\includegraphics[width=\linewidth]{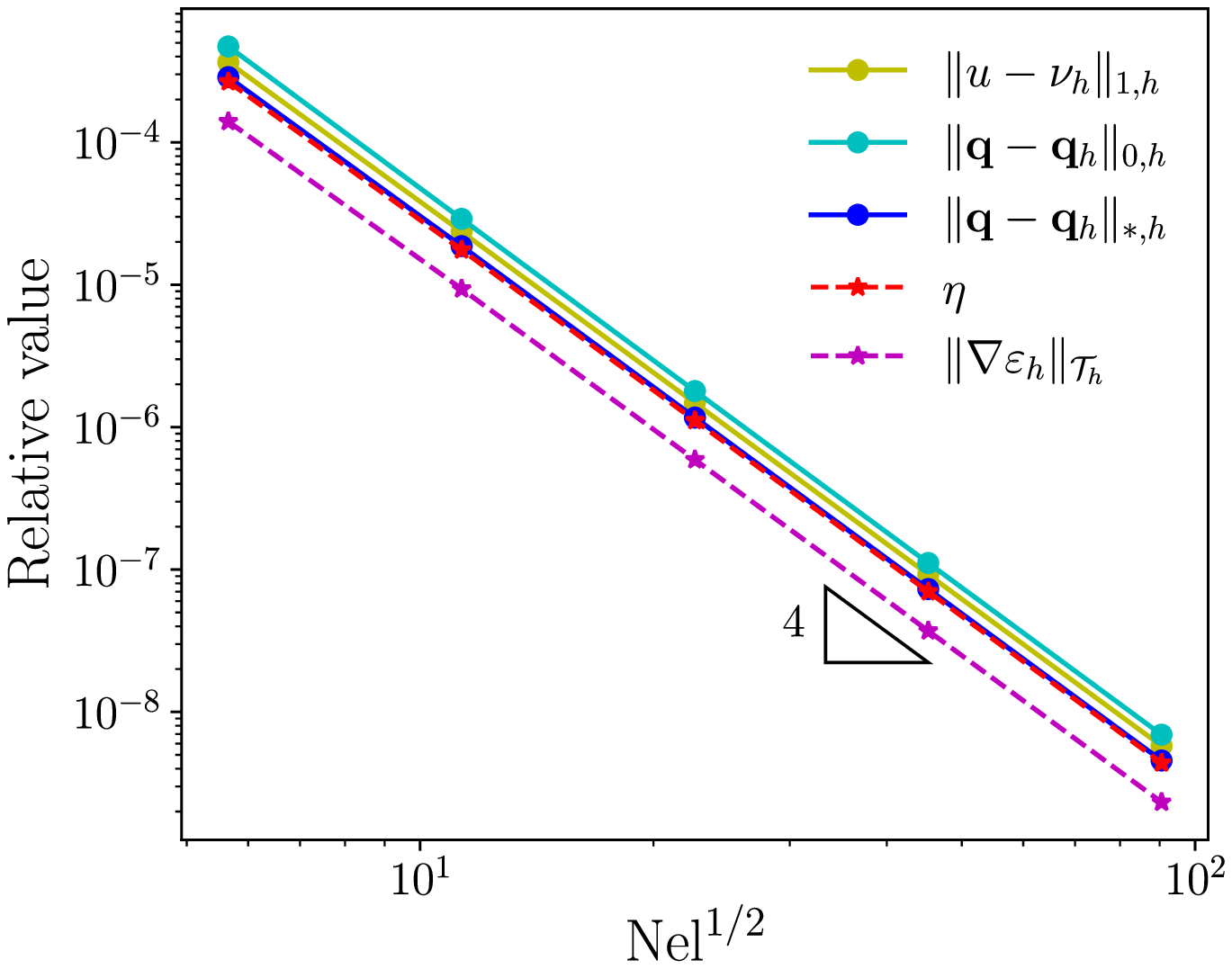}
		\caption{Error/estimator vs. $ {\rm Nel}^{1/2} $, $ p=3 $}
		\label{fig:curves_smooth_p3}
	\end{subfigure}
	\hfill
	\begin{subfigure}{0.49\linewidth}
		\centering
		\includegraphics[width=\linewidth]{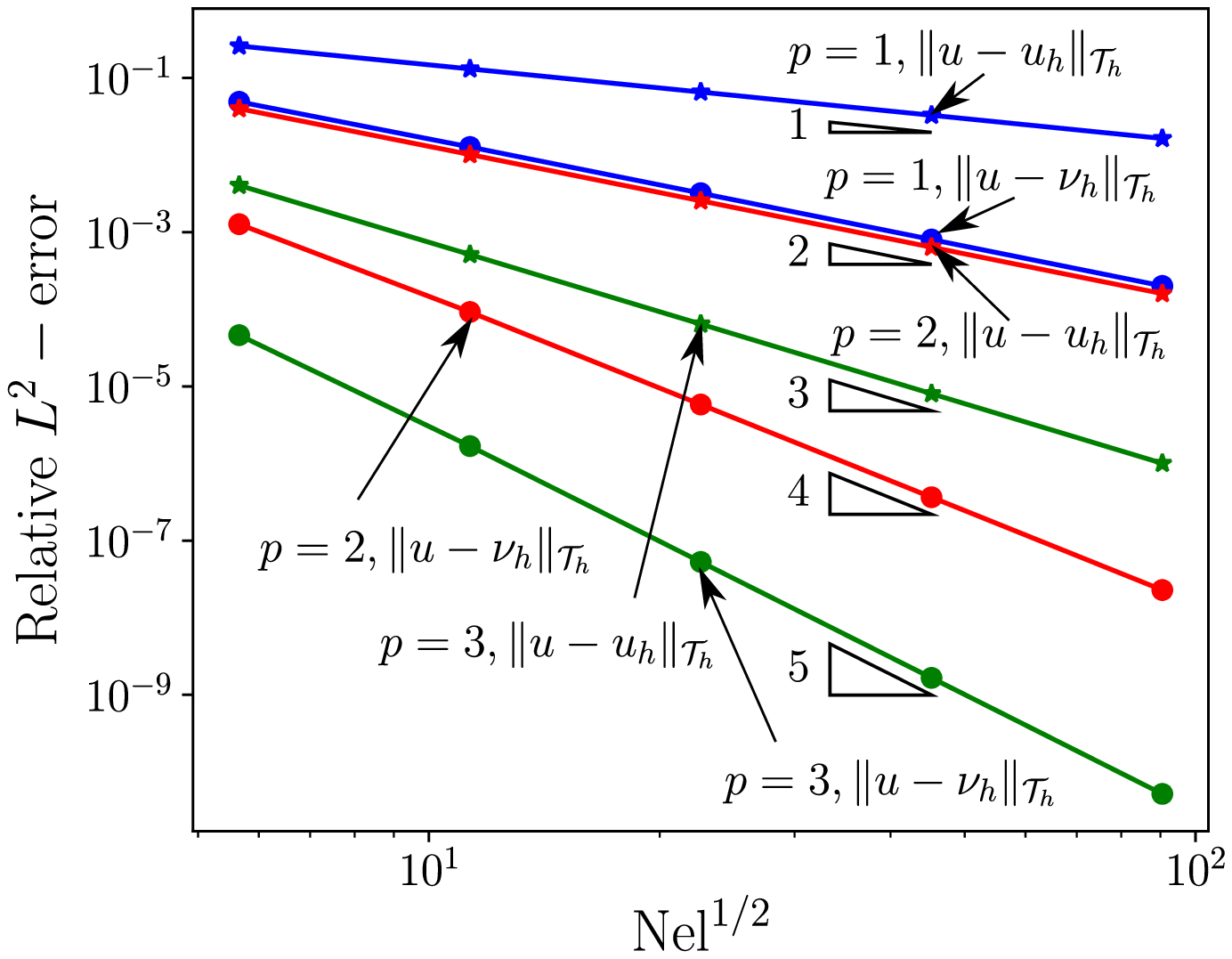}
		\caption{Increment in order of convergence}
		\label{fig:curves_smooth_L2}
	\end{subfigure}
	\caption{Convergence rates for a smooth solution scenario under uniform mesh refinements.}
	\label{fig:conv_rates_smooth}
\end{figure}
\newpage
We notice that the behavior is consistent with Remark \ref{thm:apriori_aux}, that is, we gain two orders of convergence for any integer $ p\geq 2 $, and only one order for $ p=1 $, since $ f\notin\CV_h^{p-1} $ in this case.

On another hand, Figure \ref{fig:efficiency_smooth_1} shows how the effectivity index of $ \eta $ becomes uniform with respect to $ h $, for each polynomial order $ p $ under consideration; while Figure \ref{fig:efficiency_smooth_2} displays the constant $ \delta $ involved in the saturation Assumption \ref{eq:saturation}. Notice that the value of $ \delta $ remains bounded uniformly with respect to $ h $ for each polynomial order $ p $. Observe that $ \delta $ decreases as $ p $ increases, which would allow us to conjecture two things: $ \delta < 1 $, for all $ p\in\mathbb{N} $; and, as $ p $ increases, a smaller reliability constant in \eqref{eq:reliability} is obtained (see the proof of Theorem \ref{rel} in Appendix \ref{append:rel}).

\begin{figure}[h!]
	\centering
	\begin{subfigure}{0.49\textwidth}
		\centering
		\includegraphics[width=\textwidth]{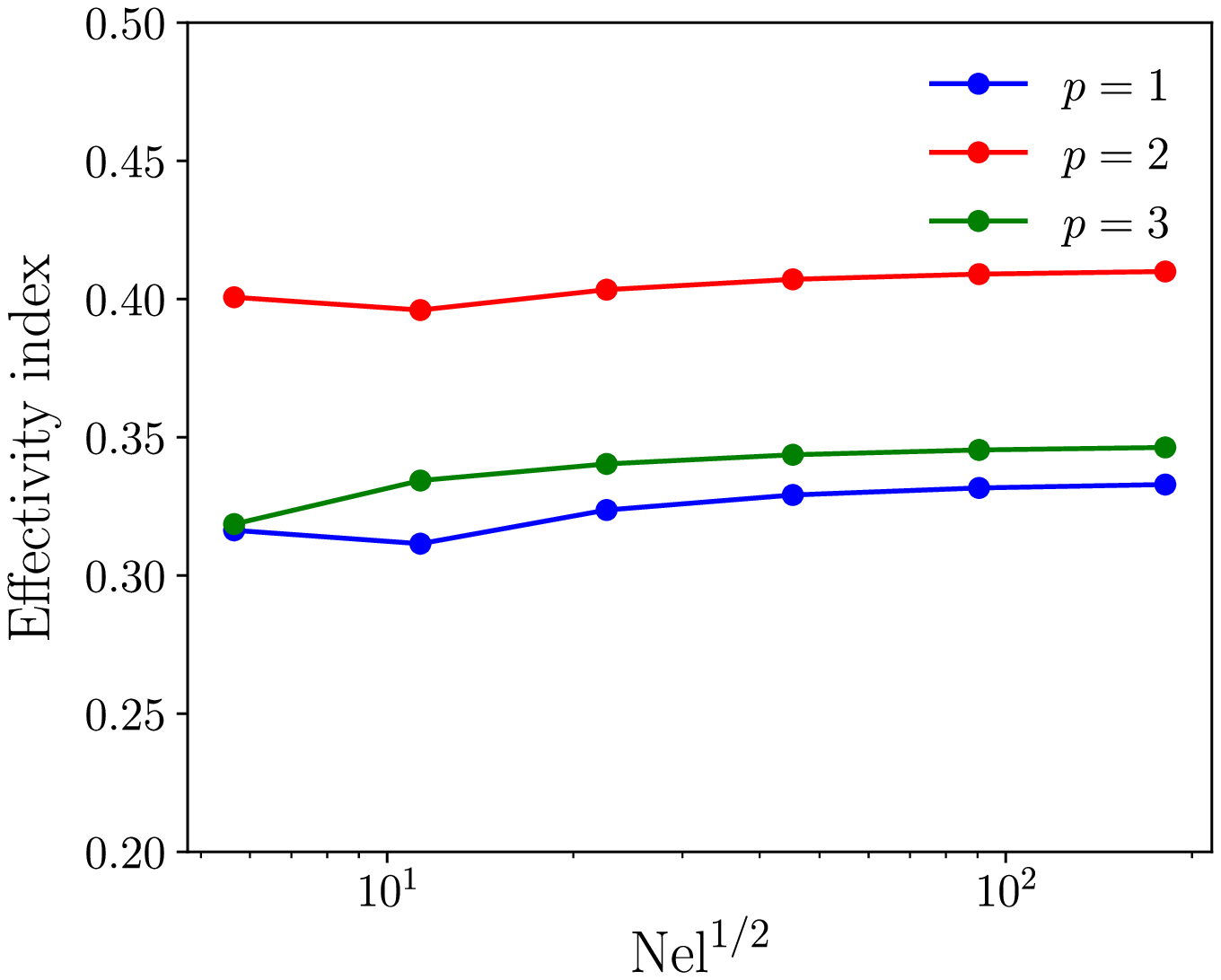}
		\caption{Effectivity index of $ \eta $ vs. $ {\rm Nel}^{1/2} $}
		\label{fig:efficiency_smooth_1}
	\end{subfigure}
	\hfill
	\begin{subfigure}{0.475\textwidth}
		\centering
		\includegraphics[width=\textwidth]{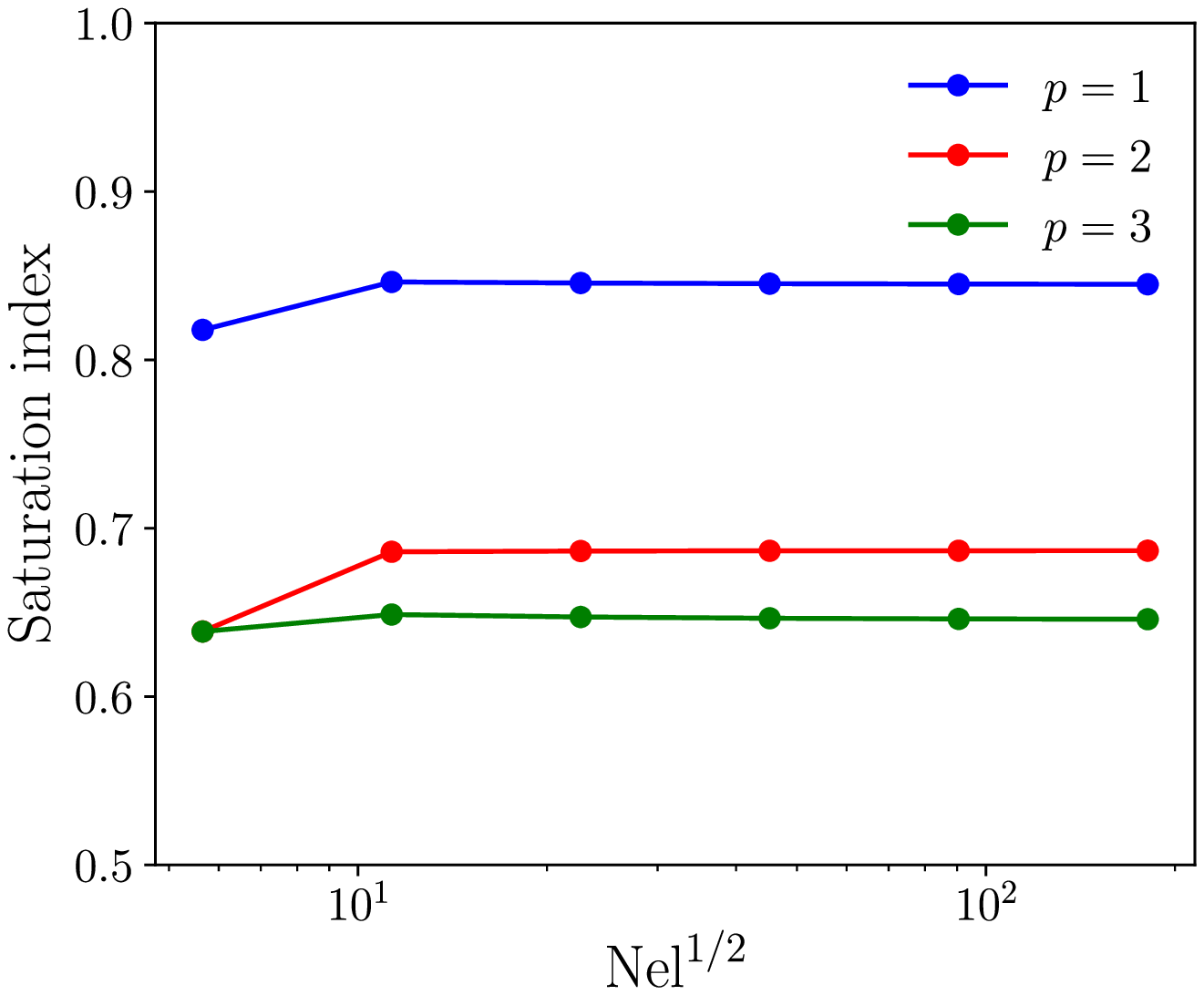}
		\caption{Saturation constant $ \delta $ vs. $ {\rm Nel}^{1/2} $}
		\label{fig:efficiency_smooth_2}
	\end{subfigure}
	\caption{Effectivity index and saturation constant for a smooth solution scenario.}
	\label{fig:efficiency_smooth}
\end{figure}

\subsection{A low-regularity solution}\label{L}
We consider again the Poisson problem \eqref{eq:model_problem}, but now on the L-shaped domain $ \Omega\eq(-1,1)^2\setminus(-1,0)^2 $. We set $ f\equiv 0 $ and $ u_D $ such that the exact solution in polar coordinates is $ u(r,\vartheta)=r^{2/3}\sin\left(\frac{2}{3}(\pi-\vartheta)\right) $. We note that $ u\in H^{2/3-\epsilon}(\Omega) $, for all $ \epsilon>0 $, due to a singularity located at the origin $ (0,0) $. 

For adaptively refined meshes, Figure \ref{fig:conv_rates_L_eta} shows the convergence curves associated with the error estimator $ \eta $; while Figure \ref{fig:conv_rates_L_H1} shows the curves associated with the full error $ (\|u-\nu_h\|_{1,h}^2+\|\bq-\bq_h\|_{0,h}^2)^{1/2} $. Notice that we recover the convergence rates that we would have in a smooth scenario, together with superconvergence.

\begin{figure}[h!]
	\centering
	\begin{subfigure}{0.49\textwidth}
		\centering
		\includegraphics[width=\linewidth]{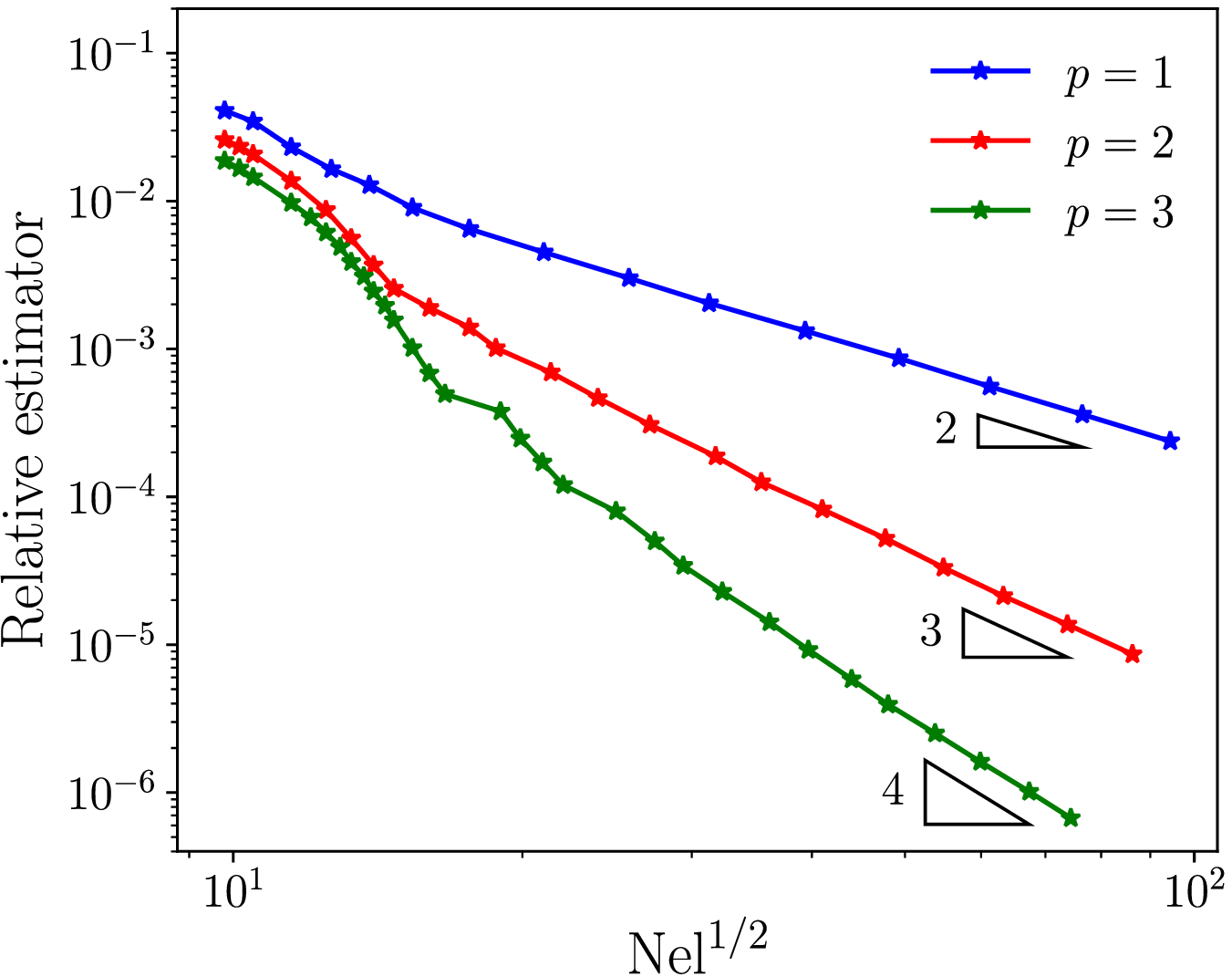}
		\caption{$ \eta $ vs. $ {\rm Nel}^{1/2} $}
		\label{fig:conv_rates_L_eta}
	\end{subfigure}
	\hfill
	\begin{subfigure}{0.49\textwidth}
		\centering	
		\includegraphics[width=\linewidth]{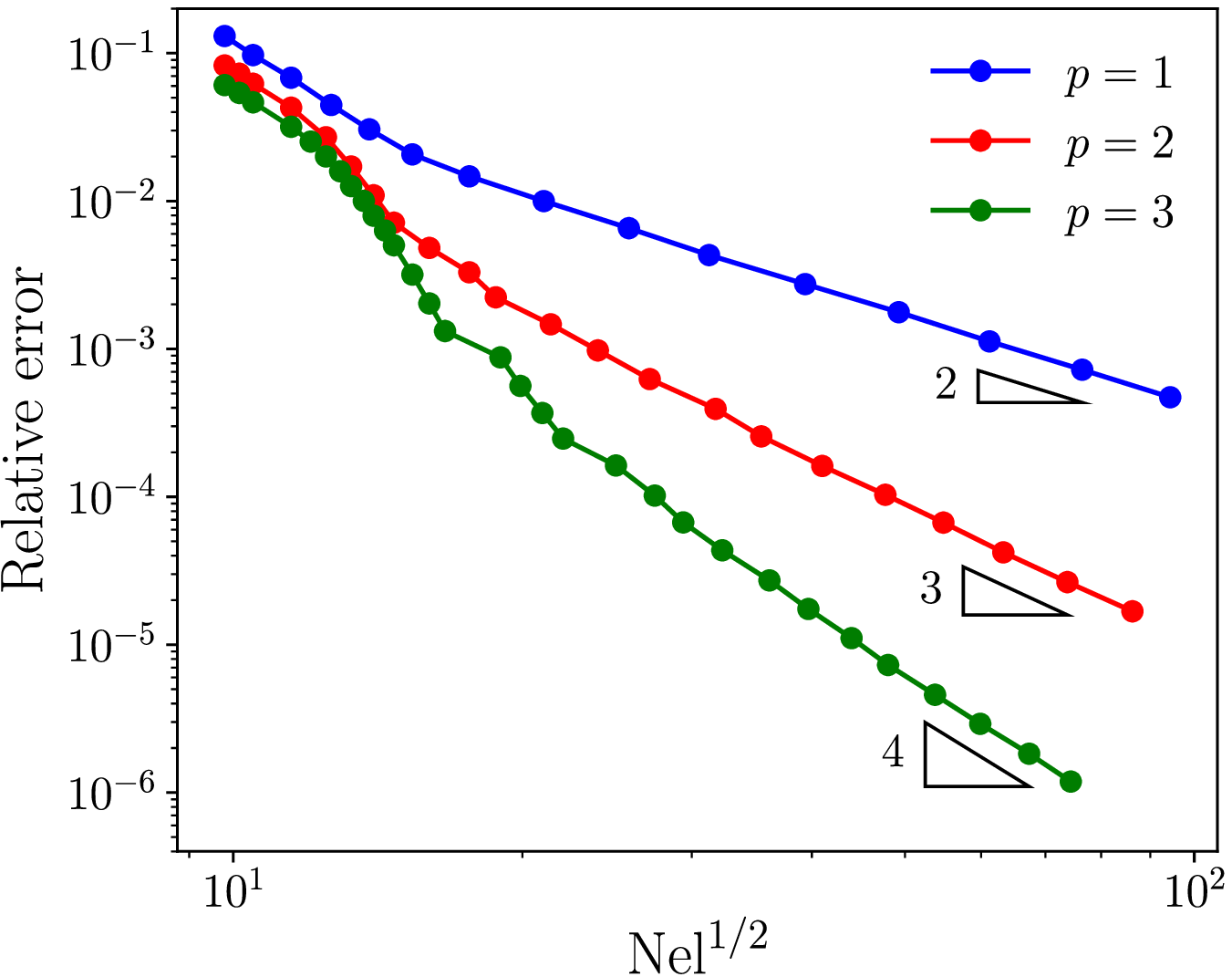}
		\caption{Full error vs. $ {\rm Nel}^{1/2} $}
		\label{fig:conv_rates_L_H1}
	\end{subfigure}
	\caption{Convergence curves for a low-regularity solution scenario.}
	\label{fig:conv_rates_L_shaped_H1_est}
\end{figure}

On the another hand, Figures \ref{fig:initial_mesh_L_shaped}, \ref{fig:mesh_5_L_shaped}, and \ref{fig:mesh_10_L_shaped} display the evolution of the adapted mesh refinement (driven by our estimator $ \eta $), and how the mesh is locally refined in a neighborhood of the corner $ (0,0) $, where the gradient of $ u $ has a singularity.

\begin{figure}[h!]
	\centering
	\begin{subfigure}{0.32\textwidth}
		\centering
		\includegraphics[width=\linewidth]{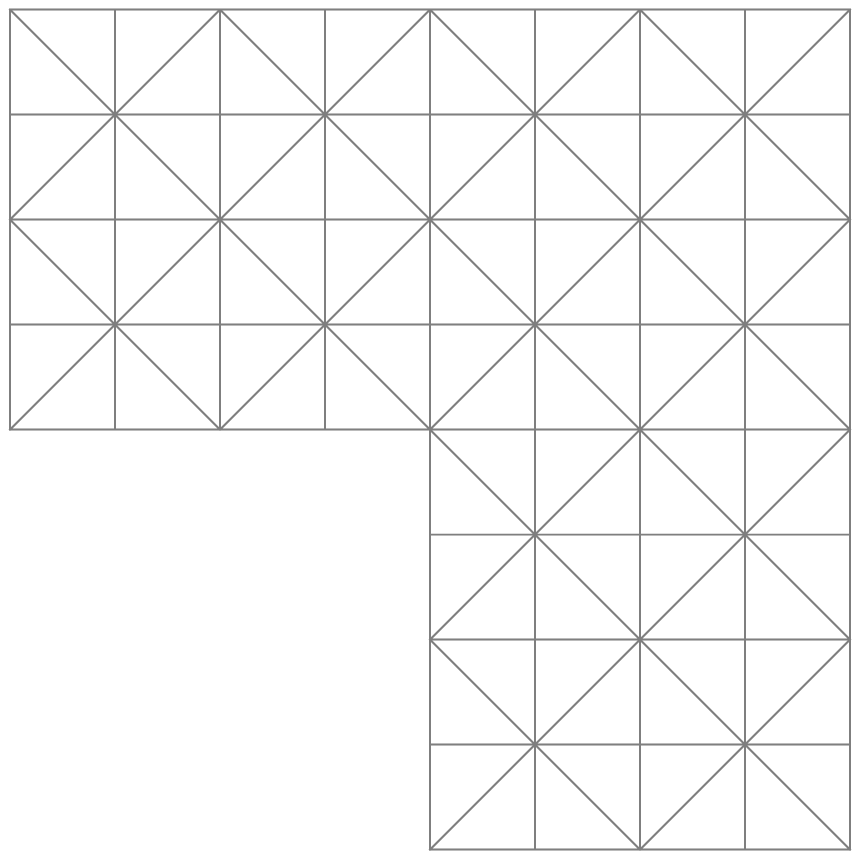}
		\caption{Initial (${\rm Nel}= 96 $)}
		\label{fig:initial_mesh_L_shaped}
	\end{subfigure}
	\hfill
	\begin{subfigure}{0.32\textwidth}
		\centering
		\includegraphics[width=\linewidth]{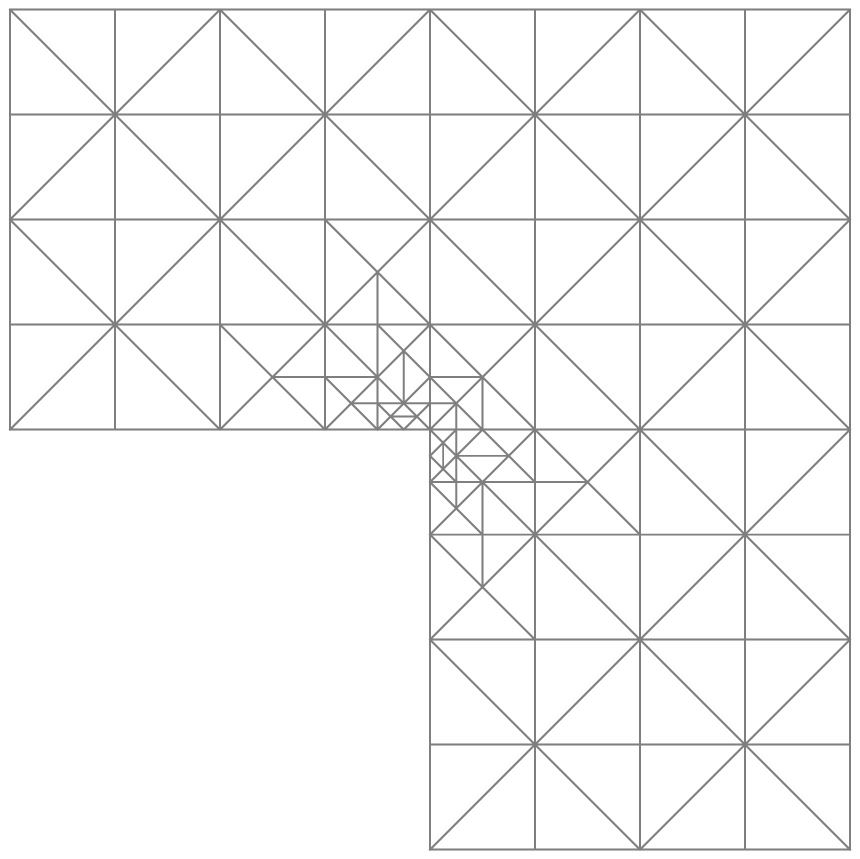}
		\caption{$ 5^{\text{th}} $ iteration (${\rm Nel}= 156 $)}
		\label{fig:mesh_5_L_shaped}
	\end{subfigure}
	\hfill
	\begin{subfigure}{0.32\textwidth}
		\centering
		\includegraphics[width=\linewidth]{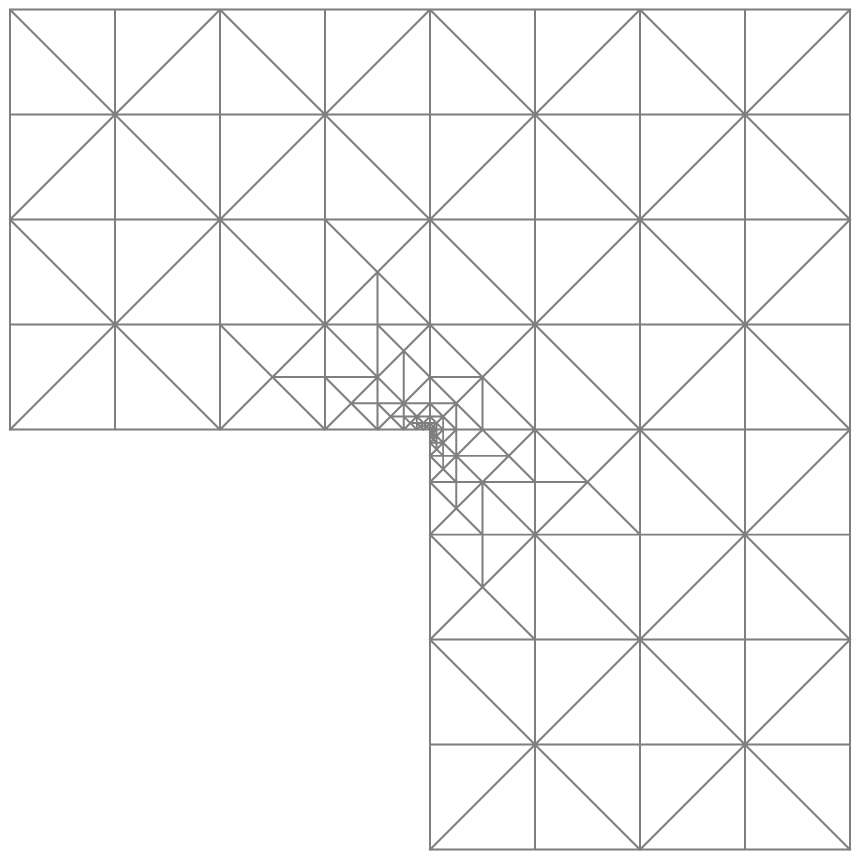}
		\caption{$ 10^{\text{th}} $ iteration (${\rm Nel}= 207 $)}
		\label{fig:mesh_10_L_shaped}
	\end{subfigure}
	\caption{Sequence of adapted meshes for a low-regularity solution scenario, with $ p=3 $.}
	\label{fig:adaptive_meshes_L_shaped}
\end{figure}

\noindent
\begin{minipage}{0.47\textwidth}\vspace{-1cm}
	Finally, in Figure \ref{fig:conv_rates_L_shaped_L2} we show the $ L^2 $-errors of the original finite element (BDM) solution $ u_h $, and the residual minimization postprocessed solution $ \nu_h $. Notice the increment in the convergence rates when we consider $ \nu_h $ instead of $ u_h $. The curves show optimal convergence rates for $ u_h $, and superconvergence for $ \nu_h $, thanks to the use of the adaptive mesh refinement process driven by $ \eta $.
	In contrast to Figure \ref{fig:curves_smooth_L2}, we gain two orders of convergence for any $ p\geq 1 $, since $ f\equiv 0\in\CV_h^{p-1} $ in this experiment.
\end{minipage}\hfill
\begin{minipage}{0.48\textwidth}
	\centering
	\includegraphics[width=\linewidth]{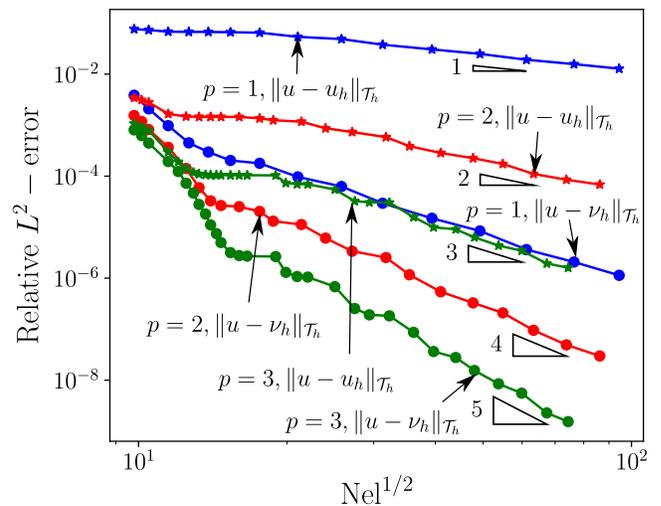}
	\captionof{figure}{Increment in order of convergence.}
	\label{fig:conv_rates_L_shaped_L2}
\end{minipage}
	
\subsection{Application to advection-diffusion problems}\label{ad_diff}
Let us consider the following advection-diffusion problem: Given $ f:\Omega\to\mathbb{R} $, $ {\boldsymbol\beta}\in\mathbb{R}^2 $, and $ u_D:\partial\Omega\to\mathbb{R} $, find $ u:\Omega\to\mathbb{R} $ such that
\begin{subequations}\label{eq:adr}
	\begin{align}
		-\Delta u+({\boldsymbol\beta}\cdot\grad)u&=f\ \ \quad\text{in }\Omega,\\
		u&=u_D\quad\text{on }\partial\Omega.
	\end{align}
\end{subequations}

\noindent
In mixed form, equation \eqref{eq:adr} becomes the system
\begin{subequations}\label{eq:adr_system}
	\begin{align}
		\bq+\grad u&={\bf 0}\ \ \quad\text{in }\Omega,\\
		\div\bq-{\boldsymbol\beta}\cdot\bq&=f\ \ \quad\text{in }\Omega,\\
		u&=u_D\quad\text{on }\partial\Omega.
	\end{align}
\end{subequations}
\noindent We set $ {\rm P}=10^3/3 $ and $ {\boldsymbol\beta}=({\rm P},{\rm P}) $. We consider $ f $ and $ u_D $ such that the exact solution to \eqref{eq:adr} is
\begin{equation*}
	u(x_1,x_2)=\left(x_1+\frac{e^{{\rm P}x_1}-1}{1-e^{\rm P}}\right)\!\!\left(x_2+\frac{e^{{\rm P}x_2}-1}{1-e^{\rm P}}\right).
\end{equation*}

We approximate the solution to \eqref{eq:adr_system} through a BDM finite element discretization. Then, we solve the residual minimization problem \eqref{eq:resmin} to postprocess the discrete solution and extract $ \|\grad\varepsilon_h\|_{\CT_h} $ to construct the improved estimator $ \eta $ (used to drive adaptive mesh refinements). Figure \ref{fig:conv_rates_ad_eta} displays the convergence curves of the error estimator $ \eta $, while Figure \ref{fig:conv_rates_ad_H1} displays the convergence curves of the full error $ (\|u-\nu_h\|_{1,h}^2+\|\bq-\bq_h\|_{0,h}^2)^{1/2} $. Despite the dominant convection regime, which usually impairs convergence, we have obtained optimal convergence rates.

\begin{figure}[h]
	\centering
	\begin{subfigure}{0.49\textwidth}
		\centering
		\includegraphics[width=\linewidth]{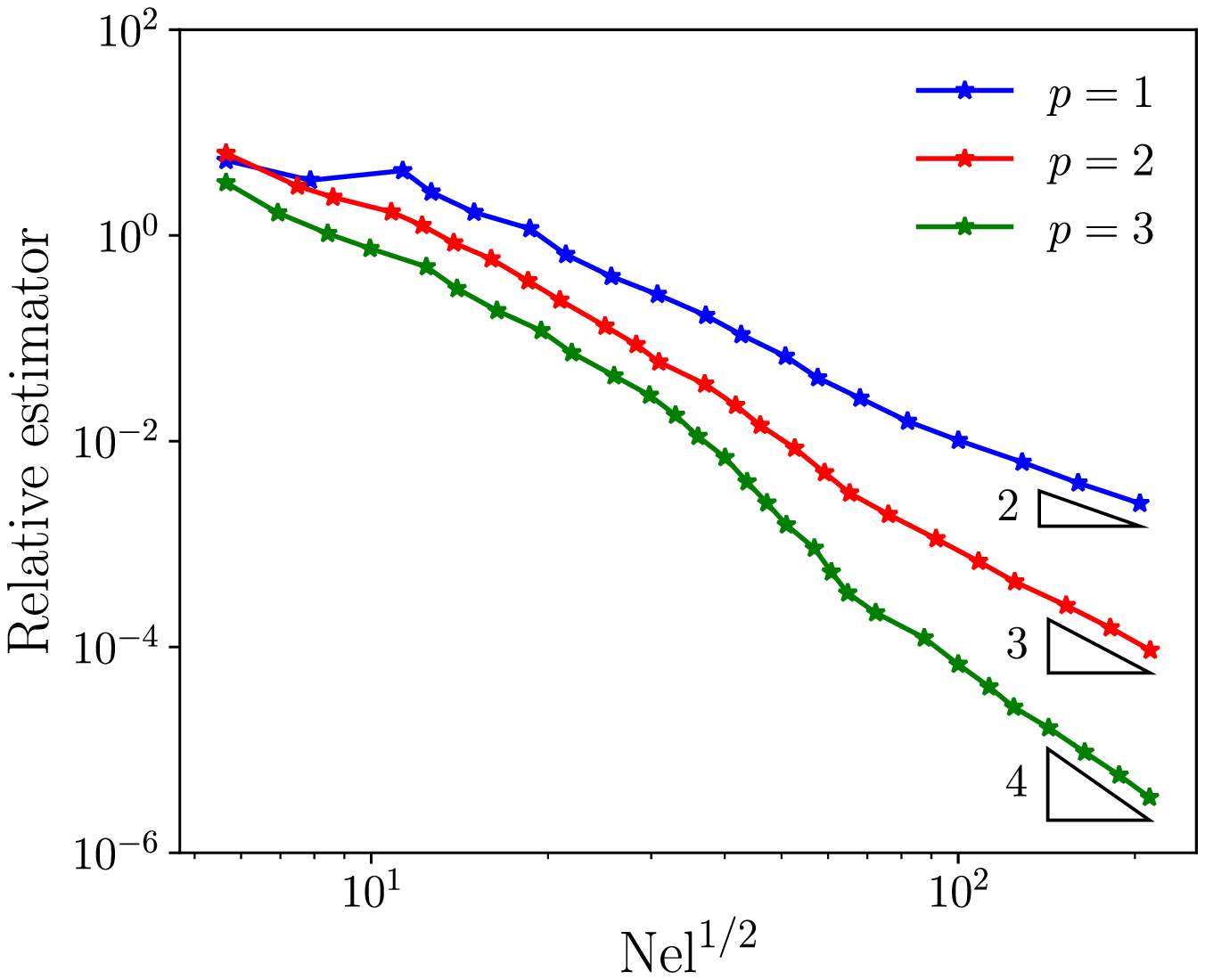}
		\caption{$ \eta $ vs. $ {\rm Nel}^{1/2} $}
		\label{fig:conv_rates_ad_eta}
	\end{subfigure}
	\hfill
	\begin{subfigure}{0.49\textwidth}
		\centering	
		\includegraphics[width=\linewidth]{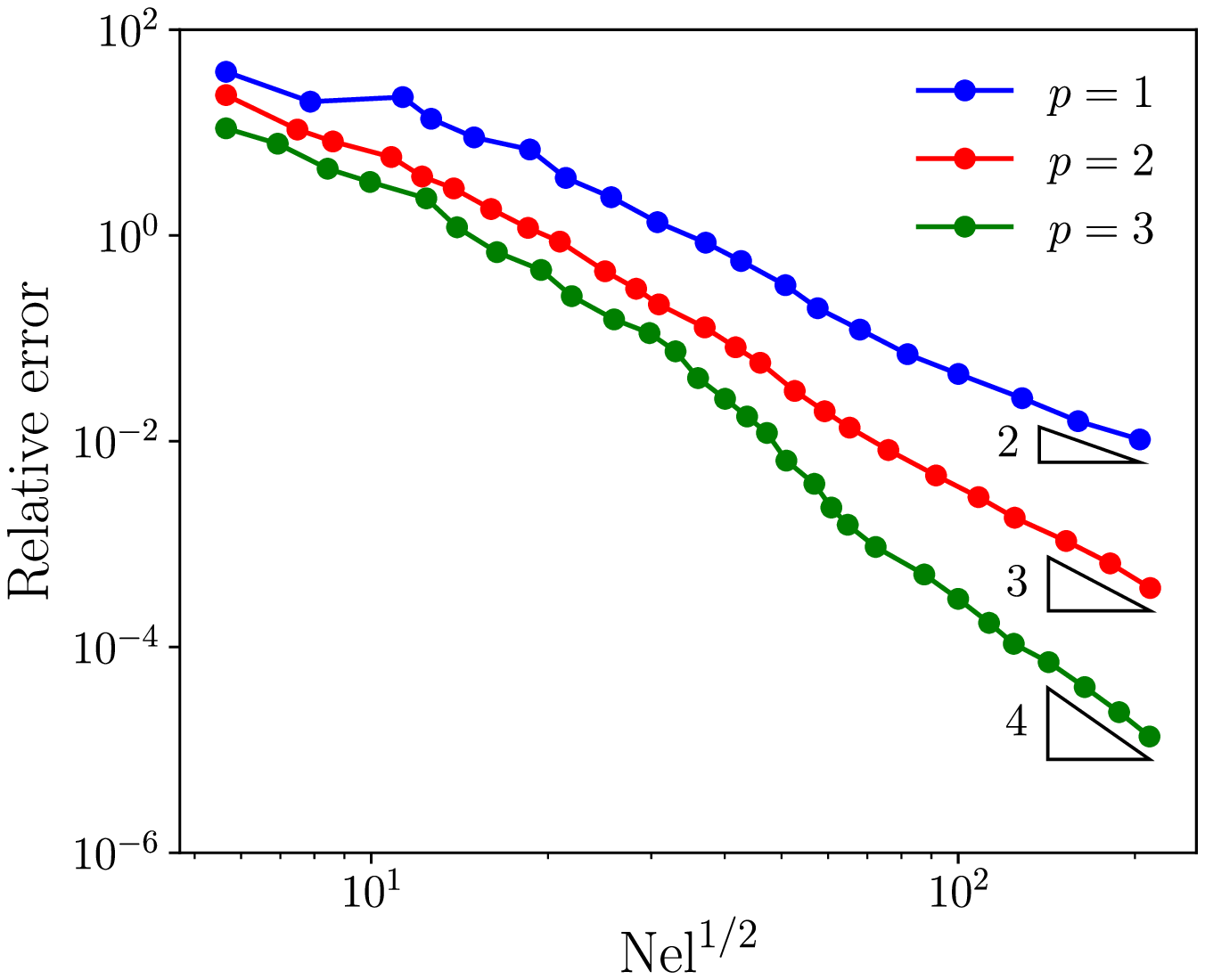}
		\caption{Full error vs. $ {\rm Nel}^{1/2} $}
		\label{fig:conv_rates_ad_H1}
	\end{subfigure}
	\caption{Convergence curves for the advection-diffusion problem.}
	\label{fig:conv_rates_ad_H1_est}
\end{figure}

On another hand, Figure \ref{fig:adaptive_meshes_ad} shows how our method delivers adaptive meshes that sharply capture the boundary layer, as we can observe from highly localized refinements around the outflow boundary.

\begin{figure}[h]
	\centering
	\begin{subfigure}{0.32\textwidth}
		\centering
		\includegraphics[width=\linewidth]{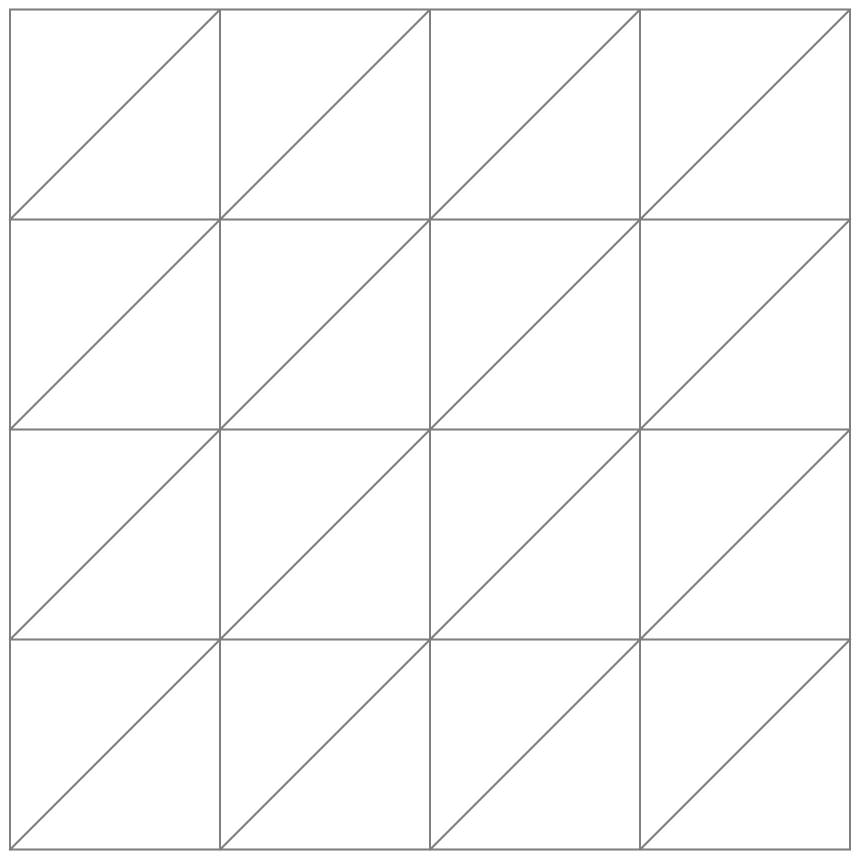}
		\caption{Initial (${\rm Nel}= 32 $)}
		\label{fig:initial_mesh_ad}
	\end{subfigure}
	\hfill
	\begin{subfigure}{0.32\textwidth}
		\centering
		\includegraphics[width=\linewidth]{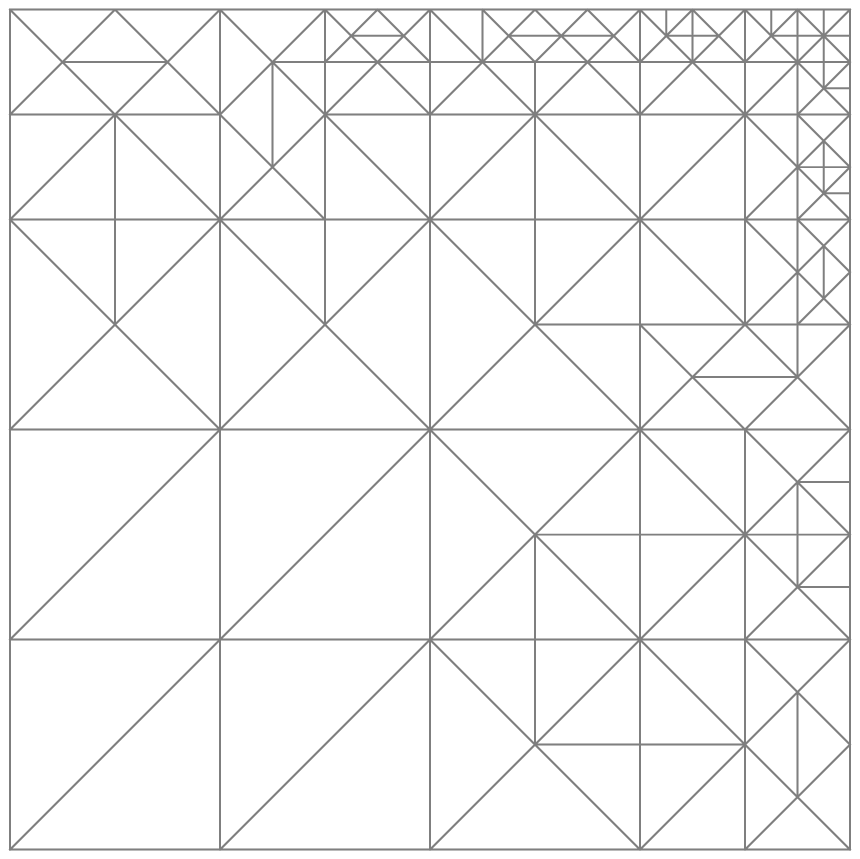}
		\caption{$ 5^{\text{th}} $ iteration (${\rm Nel}= 196 $)}
		\label{fig:mesh_5_ad}
	\end{subfigure}
	\hfill
	\begin{subfigure}{0.32\textwidth}
		\centering
		\includegraphics[width=\linewidth]{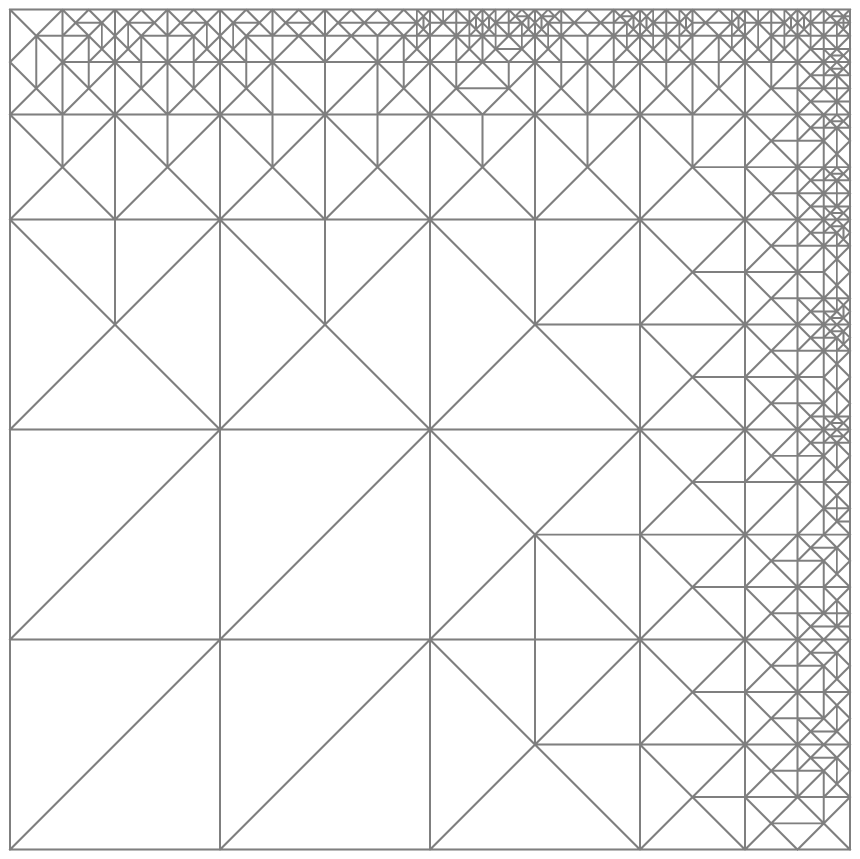}
		\caption{$ 10^{\text{th}} $ iteration (${\rm Nel}= 887 $)}
		\label{fig:mesh_10_ad}
	\end{subfigure}
	\caption{Sequence of adapted meshes for the advection-diffusion problem scenario, with $ p=3 $.}
	\label{fig:adaptive_meshes_ad}
\end{figure}

\noindent
\begin{minipage}{0.47\textwidth}\vspace{-0.5cm}
	Finally, in Figure \ref{fig:conv_rates_ad_L2} we see the contrast between the behavior of the $ L^2 $-error associated with the BDM mixed method and the postprocessing scheme. We observe the improvement in the rates of convergence for each polynomial approximation degree. In particular, notice the lack of convergence of the BDM mixed method for $ p=1 $ and $ p=2 $. which is optimally restored by our scheme. 
	As in the smooth solution experiment (see Figure \ref{fig:curves_smooth_L2}) we have that $ f\notin\CV_h^{p-1} $, so the first estimate in \eqref{apriori_minres_aux} does not hold for $ p=1 $.
\end{minipage}
\hfill
\begin{minipage}{0.48\textwidth}
	\includegraphics[width=\linewidth]{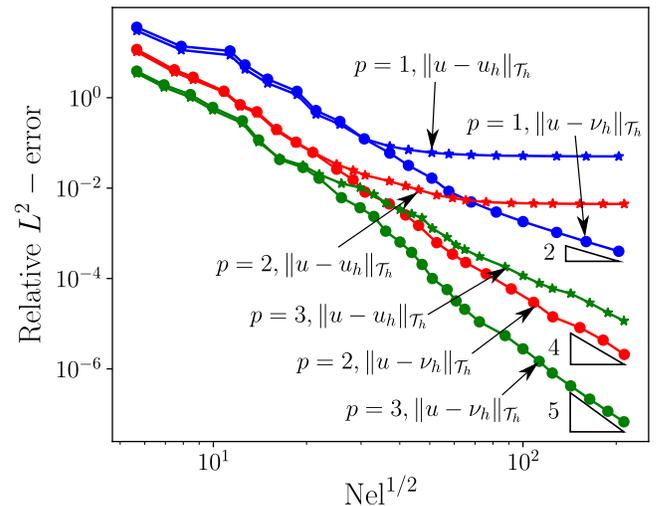}
	\captionof{figure}{Increment in order of convergence.}
	\label{fig:conv_rates_ad_L2}
\end{minipage}

\section{Conclusions}
\label{sec:conclusions}

Inspired on the postprocessing technique due to Stenberg \cite{stenberg_1991}, we have proposed an adaptive finite element method based on residual minimization, suitable for partial differential equations involving a diffusion term. The residual minimization approach is used on the postprocessing scheme associated with a mixed finite element discretization of our model problem. In particular, we have considered the BDM finite element method for simplicity. However, the extension to other mixed schemes (such as RT and HDG, to name a few), and to more complex models (such as advection-diffusion-reaction, or Stokes-like problems), is straightforward.

Our proposed method has two main advantages. First, the residual minimization approach computes a built-in residual representative, which has proven to be a reliable and efficient a posteriori error estimator in a weak norm. Second, since we perform residual minimization on a local postprocessing scheme, we can solve the underlying saddle point problem with minimal computational effort. Thus, our approach becomes attractive for problems involving different kinds of local behaviors needing adaptive mesh refinements. Our scheme provides, in addition to the residual representative, a superconvergent postprocessing of the scalar variable, which coincides with Stenberg's postprocessed solution.

We have also proposed and improved a posteriori error estimator, by adding the jump of the postprocessed solution, as well as a residual term quantifying the mismatch between the discrete flux and the gradient of the postprocessed solution. This improved estimator is proven to be reliable and efficient in a stronger norm. Moreover, like the built-in estimator, it can be used in a wide spectrum of problems involving diffusion terms, since both estimators are PDE independent.

\appendix

\section{Verification of Assumption~\ref{fortin_assumpt}}\label{appA}
For our choice of mixed FEM (i.e., BDM), we verify Assumption \ref{fortin_assumpt} for the lowest-order finite element spaces $ \CV_{*,h}^3 $ and $ \BCH_h^1 $, in dimension $ d=2 $. We start by stating the following auxiliary lemma.
\begin{lemma}
	For each $ K\in\CT_h $, let $ \{\varphi_i^K\}_{i=1}^6 $ be the set of nodal basis functions spanning $ \{\mu\in L^2(\partial K):\exists\bz\in \BCH_h^1\text{ such that }\mu=\bz|_K\cdot\bn_K\} $. There exists a (bi-orthogonal) set $ \{\psi_i^K\}_{i=1}^6\subset \CV_{*,K}^3 $ such that
	\begin{align}\label{eq:biorthog}
		\int_{\partial K}\varphi_i^K\psi_j^K=\xi_K\delta_{ij},\qquad i,j=1,\ldots,6\,,
	\end{align}
	where $ \delta_{ij} $ denotes the Kronecker delta, and $ \xi_K $ denotes a scaling constant of the element $ K $. Moreover we have the estimate:
	\begin{align}\label{eq:biorthog_bound}
		\|\psi_i^K\|_{\partial K}\lesssim\xi_K^{1/2},\qquad i=1,\ldots,6\,,
	\end{align}
	for all $ K\in\CT_h $.
\end{lemma}
\begin{proof}
	Consider the reference simplex $ \widehat{K}\eq\{(s,t)\in\mathbb{R}^2:s\in[0,1],t\in[0,1-s]\} $, and let $ \{\widehat{\varphi}_i^{\widehat{K}}\}_{i=1}^6 $ be defined by $ \widehat{\varphi}_i^{\widehat{K}}=\varphi_i^K\circ \CF_K $, where $ \CF_K:\widehat{K}\to K $ is the affine mapping satisfying $ K=\CF_K(\widehat{K}) $. We will construct $ \{\widehat{\psi}_i^{\widehat{K}}\}_{i=1}^6 $ such that
	\begin{align}\label{biort_ref}
		\int_{\partial\widehat{K}}\widehat{\varphi}_i^{\widehat{K}}\widehat{\psi}_j^{\widehat{K}}=\delta_{ij},\qquad i,j=1,\ldots,6,
	\end{align}
	and thus we will define $ \psi_i^K\eq\widehat{\psi}_i^{\widehat{K}}\circ \CF_K^{-1} $.
	
	Without loss of generality, let $ \widehat{\varphi}_1^{\widehat{K}} $, $ \widehat{\varphi}_2^{\widehat{K}} $ be the nodal basis functions associated with the face $ \widehat{F}_1\eq\{(s,0)\in\mathbb{R}^2:s\in[0,1]\} $; let $ \widehat{\varphi}_3^{\widehat{K}} $, $ \widehat{\varphi}_4^{\widehat{K}} $ be the nodal basis functions associated with the face $ \widehat{F}_2\eq\{(1-s,s)\in\mathbb{R}^2:s\in[0,1]\} $; and let $ \widehat{\varphi}_5^{\widehat{K}} $, $ \widehat{\varphi}_6^{\widehat{K}} $ be the nodal basis functions associated with the face $ \widehat{F}_3\eq\{(0,s)\in\mathbb{R}^2:s\in[0,1]\} $. We are going to exhibit three functions $ \{\widehat{\phi}_j\}_{j=1}^3\subset\CV_{*,\widehat{K}}^3 $\,, such that
	\begin{equation}\label{biort}
		\widehat{\phi}_j|_{\widehat{F}_k}=0,\qquad\text{for }j=1,2,3,\quad\text{and}\quad k=2,3.
	\end{equation}
	\noindent Indeed, for $ \widehat{\lambda}_1(\widehat{x}_1,\widehat{x}_2)\eq 1-\widehat{x}_1-\widehat{x}_2 $ and $ \widehat{\lambda}_2(\widehat{x}_1,\widehat{x}_2)\eq \widehat{x}_1 $, we set $ \widehat{\phi}_1\eq\widehat{\lambda}_1\widehat{\lambda}_2 $, $ \widehat{\phi}_2\eq\big(\widehat{\lambda}_1\big)^{\!2}\ \!\widehat{\lambda}_2 $, and $ \widehat{\phi}_3\eq\widehat{\lambda}_1\big(\widehat{\lambda}_2\big)^{\!2} $. We propose to find linear combinations
	\begin{equation}\label{linear}
		\widehat{\psi}_1^{\widehat{K}}\eq\beta_1\widehat{\phi}_1+\beta_2\widehat{\phi}_2+\beta_3\widehat{\phi}_3\qquad\text{and}\qquad\widehat{\psi}_2^{\widehat{K}}\eq\gamma_1\widehat{\phi}_1+\gamma_2\widehat{\phi}_2+\gamma_3\widehat{\phi}_3,
	\end{equation}
	such that
	\begin{equation}\label{biort_ref_12}
		\int_{\widehat{F}_1}\widehat{\varphi}_i^{\widehat{K}}\widehat{\psi}_j^{\widehat{K}}=\delta_{ij},\qquad i,j=1,2,
	\end{equation}
	and
	\begin{equation}\label{zeromean}
		\int_{\widehat{K}}\widehat{\psi}_j^{\widehat{K}}=0,\qquad j=1,2.
	\end{equation}
	Notice that \eqref{biort} and \eqref{biort_ref_12} guarantee \eqref{biort_ref}.
 
	We can condensate \eqref{linear}, \eqref{biort_ref_12} and \eqref{zeromean} as the linear systems
	\begin{align*}
		\BA\boldsymbol{\beta}=\boldsymbol{e}_1\qquad\text{and}\qquad\BA\boldsymbol{\gamma}=\boldsymbol{e}_2,
	\end{align*}
	where $ \boldsymbol{\beta}\eq(\beta_1,\beta_2,\beta_3) $, $ \boldsymbol{\gamma}\eq(\gamma_1,\gamma_2,\gamma_3) $, $ \boldsymbol{e}_1=(1,0,0) $, $ \boldsymbol{e}_2=(0,1,0) $, and the coefficients of $ \BA $ are defined by $$ A_{ij}\eq\begin{cases}
		\int_{\widehat{F}_1}\widehat{\varphi}_i^{\widehat{K}}\widehat{\phi}_j\,,& i=1,2\,;\ j=1,2,3\\
		\int_{\widehat{K}}\widehat{\phi}_j\,,& i=3\,;\ j=1,2,3.
	\end{cases}\,. $$
	We have verified that $ \BA $ is invertible, so we can determine $\boldsymbol{\beta}$ and $\boldsymbol{\gamma}$ uniquely, completing the construction in \eqref{linear}.

	Thus, we obtain \eqref{biort_ref} after constructing $ \big\{\widehat{\psi}_3^{\widehat{K}},\widehat{\psi}_4^{\widehat{K}}\big\} $ and $ \big\{\widehat{\psi}_5^{\widehat{K}},\widehat{\psi}_6^{\widehat{K}}\big\} $ with the same arguments.
	Hence, \eqref{eq:biorthog} follows from \eqref{biort_ref} by standard scaling arguments with $ \xi_K\eq|\partial K|/|\partial\widehat{K}| $.
	Finally (see \cite[Lemma 1.101]{ern2013theory}),
	\begin{align*}
		\|\psi_i^K\|_{\partial K}\lesssim\xi_K^{1/2}\|\widehat{\psi}_i^{\widehat{K}}\|_{\partial\widehat{K}}\lesssim\xi_K^{1/2},\qquad i=1,\ldots,6,
	\end{align*}
	and \eqref{eq:biorthog_bound} follows.
\end{proof}

Now, we construct the Fortin operator. Let us consider, for each $ K\in\CT_h $, the operator $ \Pi_{\partial K}:\gamma_K(H_*^1(K))\to \gamma_K(\CV_{*,K}^{p+2}) $ defined by
\begin{align*}
	\Pi_{\partial K}v\eq\sum_{j=1}^6\alpha_j^K\!(v)\,\psi_j^K,
\end{align*}
where $ \alpha_j^K\!(v)\in\mathbb{R}$ is chosen to satisfy
\begin{align*}
	\int_{\partial K}\varphi_i^K\Pi_{\partial K}v=\int_{\partial K}\varphi_i^Kv,\qquad i=1,\ldots,6.
\end{align*}
Solving for $ \alpha_i^K\!(v) $, and thanks to \eqref{eq:biorthog}, we obtain
\begin{align*}
	\alpha_i^K\!(v)=\frac{1}{\xi_K}\int_{\partial K}\varphi_i^K v,\qquad i=1,\ldots,6.
\end{align*}
Moreover, we have
\begin{align*}
	|\alpha_i^K\!(v)|\leq\frac{1}{\xi_K}\|\varphi_i^K\|_{\partial K}\|v\|_{\partial K}\leq\frac{|\partial K|^{1/2}}{\xi_K}\|v\|_{\partial K},\qquad i=1,\ldots,6.
\end{align*}
Thus, thanks to \eqref{eq:biorthog_bound} and since $ \xi_K=|\partial K|/|\partial\widehat{K}| $, we get
\begin{align*}
	\|\Pi_{\partial K}v\|_{\partial K}^2\lesssim\sum_{j=1}^6|\alpha_j^K\!(v)|^2\|\psi_j^K\|_{\partial K}^2\lesssim\frac{|\partial K|}{\xi_K}\|v\|_{\partial K}^2=|\partial\widehat{K}|\|v\|_{\partial K}^2\lesssim\|v\|_{\partial K}^2.
\end{align*}

\section{Proofs of Theorems \ref{rel} and \ref{rel_improved}}\label{appendix_rel}
\subsection{Auxiliary estimates}

We will prove here two auxiliary lemmas needed for the proofs of Theorems \ref{rel} and \ref{rel_improved}. Let $ \gamma_K $ be the Dirichlet trace operator considered in Assumption \ref{fortin_assumpt}. We start with an inequality for polynomials lying in the complementary subspace of the kernel of the restricted trace operator $ \gamma_K|_{\CV_{*,K}^{p+2}}$\,, for each $ K\in\CT_h $.

\begin{lemma}\label{aux_ineq}
	Given $ K\in\CT_h $, let $ \CV_{\partial K}\subset \CV_{*,K}^{p+2} $ be the complementary subspace of $ \ker\!\big(\gamma_K|_{\CV_{*,K}^{p+2}}\big) $. Then
	\begin{align}
		\|\grad v_K\|_K\lesssim h_K^{-1/2}\|v_K\|_{\partial K}\qquad\forall v_K\in \CV_{\partial K}.
	\end{align}
\end{lemma}

\begin{proof}
	Let $ \widehat{K} $ be the reference simplex. Since $ \|\cdot\|_{\partial\widehat{K}} $ and $ \|\grad\cdot\|_{\widehat{K}} $ are norms in the finite-dimensional vector space $ \CV_{\partial\widehat{K}}\subset \CV_{*,\widehat{K}}^{p+2} $, they are equivalent norms. Then, we have $ \|\grad v_{\widehat{K}}\|_{\widehat{K}}\lesssim\|v_{\widehat{K}}\|_{\partial\widehat{K}} $. The result follows from a standard scaling arguments (see \cite[Lemma 1.101]{ern2013theory}).
\end{proof}

\begin{lemma}\label{lemma_rel}
	Let $(u,\bq) \in \CV \times \BCH$ be the solution of the continuous problem~\eqref{eq:model_problem}; let $(u_h,\bq_h) \in \CV_h^{p-1} \times \BCH_h^p$ be the solution of the discrete problem~\eqref{eq:discrete_formulation}; let $\theta_h \in \CV_{*,h}^{p+2}$
	be the solution of \eqref{eq:aux_problem}; and let $(\varepsilon_h,\nu_h)\in \CV_{*,h}^{p+2}\times \CV_h^{p+1} $ be the solution of \eqref{eq:saddlepoint}. Then, for all $ K\in\CT_h $, the following holds true:
	\begin{align}\label{eq:reliability_prev}
		\|\grad\!\left(\theta_h-\nu_h \right)\!\|_K&= \widetilde{\eta}_K,\\
		\label{eq:reliability_prev2}
		\|\bq-\bq_h\|_{*,K}&\leq \widetilde{\eta}_K+\|\grad(u-\nu_h)\|_K.\\
		\intertext{Moreover, if Assumption \ref{fortin_assumpt} is satisfied, then}
		\label{eq:reliability_prev3}
		h_K^{1/2}\|(\bq-\bq_h)\cdot\bn\|_{\partial K}&\lesssim \widetilde{\eta}_K+\|\grad(u-\nu_h)\|_K+{\rm osc}_K(\bq),
	\end{align}
	where $ {\rm osc}_K(\bq) $ is defined in \eqref{osc}.
\end{lemma}
\begin{proof}
	Using the first equation of \eqref{eq:saddlepoint}, we get
	\begin{align*}
		(\grad (\theta_h-\nu_h), \grad v_K)_K&=-(\grad\nu_h,\grad v_K)_K+(\grad\theta_h,\grad v_K)_K\\
		&=(\grad\varepsilon_h,\grad v_K)_K
		+(\bq_h,\grad v_K)_K+(\grad\theta_h,\grad v_K)_K\\
		&=(\grad\varepsilon_h,\grad v_K)_K
	\end{align*}
	for all $ v_K\in \CV_{*,K}^{p+2} $.
	Thus,  we have
	\begin{align*}
		\|\grad(\theta_h-\nu_h)\|_K&=\sup_{v_K\in \CV_{*,K}^{p+2}\setminus\{0\}}\frac{(\grad(\theta_h-\nu_h),\grad v_K)_K}{\|\grad v_K\|_K}=\sup_{v_K\in \CV_{*,K}^{p+2}\setminus\{0\}}\frac{(\grad\varepsilon_h,\grad v_K)_K}{\|\grad v_K\|_K}\\
		&=\|\grad\varepsilon_K\|_K=\widetilde{\eta}_K,
	\end{align*}
	which proves \eqref{eq:reliability_prev}. Now, we deal with the flux error. For each $ K\in\CT_h $, we recall that
	\begin{align*}
		\|\bq-\bq_h\|_{*,K}\eq\sup_{v_K\in \CV_{*,K}^{p+2}}\frac{(\bq-\bq_h,\grad v_K)_K}{\|\grad v_K\|_K},
	\end{align*}
	and thus, since
	\begin{align*}
		(\bq-\bq_h, \grad v_K)_K
		&=(\grad\varepsilon_h,\grad v_K)_K
		-(\grad(u-\nu_h),\grad v_K)_K
	\end{align*}
	for all $ v_K\in \CV_{*,K}^{p+2} $, we have
	\begin{align*}
		\|\bq-\bq_h\|_{*,K}\leq\|\grad\varepsilon_h\|_K+\|\grad(u-\nu_h)\|_K=\widetilde{\eta}_K+\|\grad(u-\nu_h)\|_K\,,
	\end{align*}
	for all $ K\in\CT_h $, which proves \eqref{eq:reliability_prev2}. To prove \eqref{eq:reliability_prev3} notice that, for each $ K\in\CT_h $,
	\begin{align*}
		\|(\bq-\bq_h)\cdot\bn\|_{\partial K}=\sup_{v\in L^2(\partial K)\setminus\{0\}}\frac{\langle(\bq-\bq_h)\cdot\bn,v\rangle_{\partial K}}{\|v\|_{\partial K}}.
	\end{align*}
	
	\noindent Thus, we get
	\begin{align*}
		\|(\bq-\bq_h)\cdot\bn\|_{\partial K}&=\sup_{v\in L^2(\partial K)\setminus\{0\}}\frac{\langle(\bq-\bq_h)\cdot\bn,v-\Pi_{\partial K}v+\Pi_{\partial K}v\rangle_{\partial K}}{\|v\|_{\partial K}}\\
		&\leq\sup_{v\in L^2(\partial K)\setminus\{0\}}\frac{\langle(\bq-\bq_h)\cdot\bn,v-\Pi_{\partial K}v\rangle_{\partial K}}{\|v\|_{\partial K}}\\
		&\quad+\sup_{v\in L^2(\partial K)\setminus\{0\}}\frac{\langle(\bq-\bq_h)\cdot\bn,\Pi_{\partial K}v\rangle_{\partial K}}{\|v\|_{\partial K}}\\
		&\leq\sup_{v\in L^2(\partial K)\setminus\{0\}}\frac{\langle(\bq-\bq_h)\cdot\bn,v-\Pi_{\partial K}v\rangle_{\partial K}}{\|v\|_{\partial K}}\\
		&\quad+C_{\Pi}\left(\sup_{v\in L^2(\partial K)\setminus\{0\}}\frac{\langle(\bq-\bq_h)\cdot\bn,\Pi_{\partial K}v\rangle_{\partial K}}{\|\Pi_{\partial K}v\|_{\partial K}}\right)\\
		&\leq h_K^{-1/2}{\rm osc}_K(\bq)+C_{\Pi}\left(\sup_{v_K\in\gamma_K\!\big(\CV_{*,K}^{p+2}\big)\setminus\{0\}}\frac{\langle(\bq-\bq_h)\cdot\bn,v_K\rangle_{\partial K}}{\|v_K\|_{\partial K}}\right)\\
		&\leq h_K^{-1/2}{\rm osc}_K(\bq)+C_{\Pi}\left(\sup_{v_K\in\gamma_K\!\big( \CV_{*,K}^{p+2}\big)\setminus\{0\}}\frac{\langle(\bq-\bq_h)\cdot\bn,v_K\rangle_{\partial K}}{h_K^{1/2}\|\grad v_K\|_K}\right),
	\end{align*}
	thanks to \eqref{Fortin}, and Lemma \ref{aux_ineq} observing that $ \gamma_K\big(\CV_{*,K}^{p+2}\big)\setminus\{0\}=\gamma_K\big(\CV_{\partial K}\big)\setminus\{0\} $. Finally, since
	\begin{align*}
		\langle(\bq-\bq_h)\cdot\bn,v_K\rangle_{\partial K}&=(\grad\varepsilon_K,\grad v_K)_K-(\grad(u-\nu_h),\grad v_K)_K\\
		&\leq\big(\widetilde{\eta}_K+\|\grad(u-\nu_h)\|_K\big)\|\grad v_K\|_K,
	\end{align*}
	for all $ v_K\in \CV_{*,K}^{p+2} $, we can conclude that
	\begin{align*}
		h_K^{1/2}\|(\bq-\bq_h)\cdot\bn\|_{\partial K}\leq {\rm osc}_K(\bq)+C_{\Pi}\big(\widetilde{\eta}_K+\|\grad(u-\nu_h)\|_K\big).
	\end{align*}
\end{proof}

\subsection{Proof of Theorem \ref{rel}}\label{append:rel}
Now, we are ready to deal with the proof of the reliability estimate associated to the residual representative $ \varepsilon_h $.
\begin{proof}
	Thanks to \eqref{eq:reliability_prev} and after summing over all $ K\in\CT_h $, we get
	\begin{equation}\label{rel_potential_aux}
		\|\grad\!\left(\theta_h-\nu_h \right)\!\|_{\CT_h}=\widetilde{\eta}.
	\end{equation}
	After combining \eqref{rel_potential_aux} with Assumption \ref{eq:saturation} and the triangle inequality, we obtain
	\begin{equation}\label{rel_potential}
		\|\grad\!\left(u-\nu_h \right)\!\|_{\CT_h}\leq\frac{1}{1-\delta}\widetilde{\eta}.
	\end{equation}
	The result follows from \eqref{rel_potential}, and adding \eqref{eq:reliability_prev2} with \eqref{eq:reliability_prev3} over all $ K\in\CT_h $.
\end{proof}

\subsection{Proof of Theorem \ref{rel_improved}}\label{append:rel_improved} To conclude, we present the proof of the reliability estimate for our improved estimator $ \eta $.
\begin{proof}
	For each $ K\in\CT_h $, we have
	\begin{align}\label{eq:reliability_prev_alt}
		\|\grad \left(\theta_h-\nu_h \right)\|_K&\leq \eta_K,\\
		\label{eq:reliability_prev3_alt}
		h_K^{1/2}\|(\bq-\bq_h)\cdot\bn\|_{\partial K}&\lesssim \eta_K+\|\grad(u-\nu_h)\|_K+{\rm osc}_K(\bq),
	\end{align}
	thanks to \eqref{eq:reliability_prev}, \eqref{eq:reliability_prev3}, and the fact that $ \widetilde{\eta}_K\leq\eta_K $. Additionally, since  $ \|\bq-\bq_h\|_K\leq\|\bq_h+\grad\nu_h\|_K+\|\grad(u-\nu_h)\|_K $, we conclude that
	\begin{align}
		\label{eq:reliability_prev2_alt}
		\|\bq-\bq_h\|_K&\leq \eta_K+\|\grad(u-\nu_h)\|_K.
	\end{align}
	For the jump term in the error, we have
	\begin{align}
		\big\|\!\jmp{u-\nu_h}\big\|_{F_{\rm i}}&=\big\|\!\jmp{\nu_h}\big\|_{F_{\rm i}}\qquad\text{and}\qquad\|u-\nu_h\|_{F_{\rm e}}=\|u_D-\nu_h\|_{F_{\rm e}}\,,	\label{eq:reliability_jump}
	\end{align}
	for $ F_{\rm i}\in\CF_K^{\rm i} $ and $ F_{\rm e}\in\CF_K^{\rm e} $, thanks to the fact that $ \jmp{u}_F=0 $ for all $ F_{\rm i}\in\CF_K^{\rm i} $ (since $ u\in H^1(\Omega) $) and $ u=u_D $ on $ \partial\Omega $. 
	
	Finally, \eqref{eq:reliability_alt} is a direct consequence of adding \cref{eq:reliability_prev_alt,eq:reliability_prev3_alt,eq:reliability_prev2_alt,eq:reliability_jump} over all $ K\in\CT_h $, and the upper bound for $ \|\grad(u-\nu_h)\|_{\CT_h} $ provided by Theorem \ref{rel}.
\end{proof}

\section*{Acknowledgments}
The authors	want to thank Carlos G\'onzalez-Moraga and Jos\'e Hasbani for helping with preliminary numerical experiments, and Professor David Pardo for his insightful suggestions.

	\bibliographystyle{siam}
	\bibliography{muga_rojas_vega_2022a}

\end{document}